\documentclass[a4paper]{article}

\newcommand{\ifmma}[1]{{}}

\usepackage{ifthen}
\provideboolean{tableOfContents}
\setboolean{tableOfContents}{false}

\usepackage[latin1]{inputenc}
\usepackage[T1]{fontenc}
\usepackage{amsmath,amsfonts,amsthm,amssymb}
\usepackage{mathtools}
\usepackage{graphicx,caption}
\usepackage{psfrag}
\usepackage[hang,loose]{subfigure}
\usepackage{comment}
\usepackage{nicefrac}
\usepackage{bm}
\usepackage[breaklinks=true,colorlinks=true,urlcolor=black,linkcolor=black]{hyperref}
\usepackage{xspace}
\usepackage[normalem]{ulem} 
\usepackage{pgfplots,pgfplotstable}

\usepackage{siunitx}
\usetikzlibrary{plotmarks}
\usepackage{rotating}
\usepackage{marginnote}
\RequirePackage{nicefrac}
\usepackage{pdflscape}
\usepackage{enumerate}
\usetikzlibrary{arrows}
\usepackage{multicol}        

\usepackage{todonotes}
\usepackage{geometry}
\newtheorem{theorem}{Theorem}[section]

\newtheorem{remark}[theorem]{Remark}

\newcommand{\eps}{\varepsilon} 

\newcommand{\IR}{{\mathbb R}}
\newcommand{\IC}{{\mathbb C}}
\newcommand{\IN}{{\mathbb N}}

\newcommand{\ve}{{\mathbf e}}
\newcommand{\vf}{{\mathbf f}}
\newcommand{\vh}{{\mathbf h}}
\newcommand{\vn}{{\mathbf n}}

\newcommand{\vx}{{\mathbf x}}
\newcommand{\vu}{{\mathbf u}}
\newcommand{\vv}{{\mathbf v}}
\newcommand{\vw}{{\mathbf w}}

\newcommand{\vH}{{\mathbf H}}

\newcommand{\vJ}{{\mathbf J}}

\newcommand{\vV}{{\mathbf V}}

\DeclareMathOperator{\realpart}{Re}
\renewcommand{\Re}{{\realpart}}

\newcommand{\imag}{\ensuremath{{\rm i}}}

\newcommand{\wt}{\widetilde}

\newcommand{\de}[1]{\partial_{#1}}

\DeclareMathOperator{\Div}{div}

\newcommand{\plcurl}{\operatorname{{\bf curl}}_{2D}}
\newcommand{\scurl}{\operatorname{curl}_{2D}}

\newcommand{\zerobf}{\boldsymbol{0}}

\newcommand{\td}{\tau}

\newcommand{\Gammat}{\Gamma}

\newcommand{\BL}{{\mathrm{BL},\eps}}

\newcommand{\ie}{\textit{i.\,e\mbox{.}}\xspace}

\tikzset{dash dotted/.style={dash pattern=on 1pt off 4pt on 6pt off 4pt}} %

\numberwithin{equation}{section}

\definecolor{green07}{rgb}{0, 0.7, 0}
\definecolor{green08}{rgb}{0, 0.8, 0}

\definecolor{byzantium}{rgb}{0.44, 0.16, 0.39}
\definecolor{burgundy}{rgb}{0.8, 0.0, 0.13}
\definecolor{chocolate}{rgb}{0.48, 0.25, 0.0}
\definecolor{indigo}{rgb}{0.29, 0.0, 0.51}
\definecolor{darkgreen}{rgb}{0.0, 0.55, 0.0}
\begin{document}
\begin{center} {\Large{Multiharmonic analysis for nonlinear acoustics with different scales}}\\[0.5cm]

{\large{Anastasia Th\"ons-Zueva$^{a,\,b}$, Kersten Schmidt$^{a,\,b,\,c}$, Adrien Semin$^{a,\,b}$}}\\[0.5cm]

{\small $a$: Research center Matheon, 10623 Berlin, Germany}\\
{\small $b$: Institut f\"ur Mathematik, Technische Universit\"at
  Berlin, Stra{\ss}e des 17. Juni 136, 10623 Berlin, Germany} \\
{\small $c$: Brandenburgische Technische Universit\"at Cottbus-Senftenberg, Institut f\"ur Mathematik, 
  Platz der deutschen Einheit~1, 03046 Cottbus, Germany} \\
\end{center}

{\small \noindent  \textbf{Corresponding author:} Anastasia Th\"ons-Zueva, Institut f\"ur Mathematik, Technische Universit\"at Berlin, Berlin, Germany\\
Address: Technische Universität Berlin,
Sekretariat MA 6-4,
Straße des 17. Juni 136,
D-10623 Berlin\\
E-mail: zueva@math.tu-berlin.de\\
Tel: +49 (0)30 314 - 25192}\\[0.5cm]


\textbf{Abstract}\\
The acoustic wave-propagation without mean flow and heat flux can be
described in terms of velocity and pressure by the compressible
nonlinear Navier-Stokes equations, where boundary layers appear at
walls due to the viscosity and a frequency interaction appears, \ie\
sound at higher harmonics of the excited frequency $\omega$ is
generated due to nonlinear advection. %
We use the multiharmonic analysis to derive asymptotic expansions for
small sound amplitudes and small viscosities both of order $\eps^2$ in
which velocity and pressure fields are separated into far field and
correcting near field close to walls and into contributions to the
multiples of $\omega$. %
Based on the asymptotic expansion we present approximate models for
either the pressure or the velocity for order $0$, $1$ and $2$, in
which impedance boundary conditions include the effect of viscous
boundary layers and contributions at frequencies $0$ and
$2\cdot\omega$ depend nonlinearly on the approximation at frequency
$\omega$.  In difference to the Navier-Stokes equations in time
domain, which has to be resolved numerically with meshes adaptively
refined towards the wall boundaries and explicit schemes require the
use of very small time steps, the approximative models can be solved
in frequency domain on macroscopic meshes. %
We studied the accuracy of the approximated models of different orders
in numerical experiments comparing with reference solutions in
time-domain.



\numberwithin{equation}{section}

\ifthenelse{\boolean{tableOfContents}}{%
  \tableofcontents
}{}%


\section{Introduction}

In this article we continue investigating the acoustic equations in the
framework of Landau and Lifschitz~\cite{Landau:1959} as a perturbation
of the Navier-Stokes equations around a stagnant uniform
fluid where heat flux is not taken into account. %
The aim of this study is to take into account nonlinear advection behaviour as
well as viscous effects in the boundary layer near rigid walls.
The governing equations in \textit{time domain} similar to the works of 
Tam et al.~\cite{Tam.Kurbatskii:2000,Tam.Kurbatskii.Ahuja.Gaeta:2001}, but for the case
of isothermal process, \ie\ 
pressure over density is constant over space, may be written as 
\begin{subequations}
\label{eq:navier.stokes:orig}
\begin{align}
  \partial_{t}\vv +(\vv\cdot\nabla)\vv  + \nabla p - \nu\Delta \vv %
  &=  \vf,%
  && \text{in }\Omega,
  \label{eq:navier.stokes:orig:mom} \\ 
  \partial_{t} p + c^2\,\Div \vv + {\Div(p\,\vv)} &= 0,%
  &&\text{in }\Omega, \label{eq:navier.stokes:orig:pr}\\
  \label{eq:navier.stokes:orig:bound}
  \vv &= \zerobf, && \text{on } \partial\Omega.
\end{align}
\end{subequations}
where $\vv$ is the acoustic velocity, $p=p'/\rho_0$ with $p'$ being the acoustic pressure
and $\rho_0 > 0$ being mean density, 
$c$ is the speed of sound, and $\nu > 0$ is the kinematic viscosity.
The introduction of $p$ instead of $p'$ is only to simplify the equations by removing the constant $\rho_0$ %
and we will regard $p$ as pressure and approximations to $p$ as pressure approximations.
In the \textit{momentum equation}~\eqref{eq:navier.stokes:orig:mom} with
some known source term $\vf$ the viscous dissipation in the momentum as well as 
the advection nonlinear term
are not neglected as we consider near wall regions where the derivatives of the 
acoustic velocity are rapidly increasing and might be crucial. %
The \textit{continuity equation}~\eqref{eq:navier.stokes:orig:pr} relates the
acoustic pressure to the divergence of the acoustic velocity. %
The system is completed by \textit{no-slip} boundary conditions. %

For gases the viscosity $\nu$ is very small and leads to
\textit{viscosity boundary layers} close to walls. These boundary layers are difficult 
to resolve in direct numerical simulations. Nevertheless, they have an essential influence
on the absorption properties. 
Mainly based on experiments the physical community has introduced slip boundary
conditions for the tangential component of the velocity, also known as wall laws,
see for example~\cite{Iftimie.Sueur:2010,Rienstra.Darau:2011,Rienstra:2006}.
For gases with small viscosity the Helmholtz equation can be completed
by viscosity dependent boundary conditions~\cite{Auregan.Starobinski.Pagneux:2001} 
to obtain an approximation of high accuracy for which the boundary layers do not
have to be resolved by finite element meshes~\cite{IhlenburgBook}. %

%

In the earlier works we studied the linear acoustic equations taking into account viscous effects 
in the boundary layer near rigid walls. In~\cite{Schmidt.Thoens.Joly:2014} we derived a complete 
asymptotic expansion for the problem based on the technique of multiscale
expansion in powers of $\sqrt{\eta}$, where $\eta$ is the dynamic viscosity
{and $\nu = \eta/\rho$}. 
This asymptotic expansion was rigorously justified with optimal error estimates.
In~\cite{Schmidt.Thoens:2014} we proposed and justified (effective)
impedance boundary conditions for the velocity as well as the pressure for possibly curved 
boundaries.


In case of stable periodic oscillations in nonlinear dynamical systems the
harmonic balance principal is used~\cite{Nayfeh.Mook:1995, Szemplinska-Stupnicka:1990, 
Weeger.Wever.Simeon:2014}.
It is described as a linear combination of a wave of the excitation frequency and its harmonics.
The referred method was presented as multiharmonic analysis for the modelling of nonlinear magnetic 
materials
~\cite{Bachinger.Langer.Schoeberl:2005,Bachinger.Langer.Schoeberl:2006}. Its stability has been 
demonstrated within the eddy current model.
For nonlinear Hamiltonian systems with a simple oscillator a special case of multiharmonic 
analysis, 
the so-called 
modulated Fourier expansion~\cite{Hairer.Lubich:2013}, has been well developed. 
The multiharmonic analysis as a method in the frequency domain is especially attractive for 
nonlinear acoustics. 
This application has not previously been investigated, either numerically or with asymp\-to\-tic 
expansions.
%
%
%
In the current work we restrict ourselves to the case of small sound amplitudes that are of order $O({\nu})$.

{The article is ordered as follows. In Sec.~\ref{sec:Main} 
we introduce a frequency domain system for the quasi-stationary solution of the nonlinear acoustic wave propagation problem %
using the multiharmonic analysis. Moreover, the main ideas of the multiscale expansions
separating far field and boundary layer contributions are introduced, that lead to the 
effective systems with impedance boundary conditions, both for the velocity and the pressure,
that are finally introduced as main results of the paper. %
The far field and boundary layer terms of the multiscale expansion and the effective systems are derived in Sec.~\ref{sec:derivation}.
Finally, in Sec.~\ref{sec:numerics} we verify the effective systems by numerical computations using high-oder finite elements.
}


\section{{Multiharmonic analysis, multiscale expansion and approximative models}}
\label{sec:Main}

\subsection{Multiharmonic analysis for the nonlinear system}
\label{sec:Multiharmonic}
In many acoustic applications the source is of one single frequency 
$\omega > 0$ and so of 
{the form 
  \begin{equation*}
    \vf(t,\vx) = \frac12 \left( \vf(\vx)\,\exp(-\imag\omega
    t) + \overline{\vf(\vx)} \exp(\imag \omega t)\right),
  \end{equation*}
with $\vf:
  \Omega \to \IC^2$ being complex valued.}
Then, we assume that the solution $(\vv,p)$ 
of~\eqref{eq:navier.stokes:orig} tends to a 
quasi-stationary solution that is periodic in time with a period 
$T=\tfrac{2\pi}{\omega}$\, which we denote by $(\vv,p)$ again, \ie
\begin{equation}
  \label{eq:quasistationar}
 \vv(t+\tfrac{2\pi}{\omega},\vx) = \vv(t,\vx), \quad
 p(t+\tfrac{2\pi}{\omega},\vx) = p(t,\vx).
\end{equation} 
This does not mean in general that one obtains a mono-frequency
solution of the same form as the 
source, but {as the problem involves only linear and quadratic terms its} solution 
can be written as combination of all the harmonics {$\cos(k \omega t), k=0,1,\ldots$ and $\sin(k \omega
t), k=1,2,\dots$. }
\begin{equation}
  \label{eq:ansatz:multiharmonic}
  \begin{aligned}
    \vv(t,\vx) &= \vv_0(\vx) + \frac12 \sum_{k=1}^\infty {\vv}_k(\vx)
    \exp(-\imag k\omega t) + \overline{\vv_k(\vx)} \exp(\imag k\omega t), \\
    p(t,\vx) &= p_0(\vx) + \frac12 \sum_{k=1}^\infty {p}_k(\vx)
    \exp(-\imag k\omega t) + \overline{p_k(\vx)} \exp(\imag k\omega t)\ ,
  \end{aligned}
\end{equation}
which 
is called
\textit{multiharmonic 
ansatz}~\cite{Bachinger.Langer.Schoeberl:2005,Bachinger.Langer.Schoeberl:2006,
Bachinger:2003}.
Inserting expansion~\eqref{eq:ansatz:multiharmonic}
 into the time-dependent problem 
\eqref{eq:navier.stokes:orig} {and identifying the terms
  corresponding to $\exp(-\imag k \omega t)$, $k=0,1,\ldots$ } 
leads to {the infinite} system of non-linear equations in space and frequency domain
\begin{subequations}
\begin{align}
  \label{eq:PDE:freq:PDE}
  {\mathcal{L}_k}(\vV, P)(\vx) + {\mathcal{N}_k}(\vV, P)(\vx) &= 
  \begin{pmatrix} 
    \vf(\vx)\\ 0
  \end{pmatrix}\delta_{k=1} \quad \text{ in }\Omega, 
  \qquad {k=0,1,\ldots}\ ,\\
  {\vV} &= {\zerobf \hspace{6em} \text{ on }\partial\Omega\ ,}
  \label{eq:PDE:freq:bc}
\end{align}
\label{eq:PDE:freq}
\end{subequations}
with the vectors $\vV = \big({\vv}_0, {\vv}_1, \dots\big)^\top$
and $P = \big({p}_0, {p}_1, \dots\big)^\top$,
collecting the coefficients of the {Fourier}
 ansatz~\eqref{eq:ansatz:multiharmonic}, 
the linear differential operators
\begin{align*}
  {\mathcal{L}}_k(\vV, P) &= \begin{pmatrix}
      &-\imag k \omega\vv_k -\nu \Delta\vv_k + \nabla p_k\\[1ex]
      &-\imag k \omega p_k + c^2\Div\vv_k
    \end{pmatrix}\ ,
\end{align*}
and the nonlinear differential operators
  \begin{align*}
    {\mathcal{N}}_0(\vV, P) & = 
    \frac 12 
    \begin{pmatrix}
      ({\vv}_0\cdot\nabla) {\vv}_{0} \\[1ex]
      \Div (\vv_0 \,p_{0})
    \end{pmatrix}
    + \frac 14 \sum_{m=0}^\infty
    \begin{pmatrix}
    ({\vv}_m\cdot\nabla) \overline{\vv_{m}}
    + (\overline{\vv_{m}} \cdot\nabla){\vv}_{m} \\[1ex]
    \Div (\vv_m \,\overline{p_{m}}
    + \overline{\vv_{m}}\, p_m)
    \end{pmatrix} \\
    {\mathcal{N}}_k(\vV, P) & = 
    \frac 12 \sum_{m=0}^k
    \begin{pmatrix}
      ({\vv}_m\cdot\nabla) {\vv}_{k-m} \\[1ex]
      \Div (\vv_m \,p_{k-m})
    \end{pmatrix}
    + \frac 12 \sum_{m=k}^\infty
    \begin{pmatrix}
    ({\vv}_m\cdot\nabla) \overline{\vv_{m-k}}
    + (\overline{\vv_{m-k}} \cdot\nabla){\vv}_{m} \\[1ex]
    \Div (\vv_m \,\overline{p_{m-k}}
    + \overline{\vv_{m-k}}\, p_m)
    \end{pmatrix}, \quad k > 0.
  \end{align*}%
{
  \begin{theorem}
    \label{theo:existence_uniqueness}
    The existence and uniqueness of a 
    quasi-stationary solution $(\vv,p)$ of~\eqref{eq:navier.stokes:orig} 
    satisfying~\eqref{eq:quasistationar}
    is equivalent to the existence and uniqueness of a solution
    $(\vv_0,\vv_1,\dots,p_0,p_1,\dots)$ of~\eqref{eq:PDE:freq}.
  \end{theorem}
  \begin{proof}
     First, let $(\vv,p)$ be the unique solution
    of~\eqref{eq:navier.stokes:orig}. Defining the vector
    $(\vv_0,\vv_1,\dots,p_0,p_1,\dots)$ by
    \begin{subequations}
    \label{eq:coefficients}
    \begin{align}
      & \vv_0(\vx) = \frac{\omega}{2\pi} \int_0^{2\pi/\omega} \vv(t,\vx)
      dt, \quad \vv_k(\vx) = \frac{\omega}{\pi} \int_0^{2\pi/\omega}
      \vv(t,\vx) \exp(\imag k \omega t)
      dt, \quad k \geqslant 1, \\
      & p_0(\vx) = \frac{\omega}{2\pi} \int_0^{2\pi/\omega} p(t,\vx)
      dt, \quad p_k(\vx) = \frac{\omega}{\pi} \int_0^{2\pi/\omega}
      p(t,\vx) \exp(\imag k \omega t)
      dt, \quad k \geqslant 1.
    \end{align}
    \end{subequations}
    the decomposition~\eqref{eq:ansatz:multiharmonic} holds and
    by construction $(\vv_0,\vv_1,\dots,p_0,p_1,\dots)$ is solution
    of~\eqref{eq:PDE:freq}. Moreover, \cite[Lemma~3.3]{Bachinger:2003} implies that the vector is unique.
  
    Now, let $(\vv_0,\vv_1,\dots,p_0,p_1,\dots)$ be the unique solution of~\eqref{eq:PDE:freq}.
    It is easy to see that $(\vv,p)$ defined by~\eqref{eq:ansatz:multiharmonic} 
    satisfies~\eqref{eq:navier.stokes:orig}. %
    If~\eqref{eq:navier.stokes:orig} would admit another solution $(\vw,q)$, then its coefficients
    $\vw_k$, $q_k$ defined similarly as in~\eqref{eq:coefficients} would fulfil~\eqref{eq:navier.stokes:orig}
    as well, which is a contradiction to the assumption of unicity.
  \end{proof}
}



\subsection{Asymptotic ansatz 
for small sound amplitude and viscosity}
\label{sec:Asymptotic}
To investigate acoustic velocity $\vv$ and acoustic pressure $p$ for small acoustic excitation and 
viscosity we introduce a small parameter $\eps \in \IR^+$ and replace the acoustic source 
$\vf$ by {$\eps^2\,\sum_{j=0}^\infty \eps^j\vf_j$, where each term $\vf_j$ is independent of $\eps$,} and the viscosity $\nu$ by $\eps^2\nu_0$ with $\nu_0 
\in \IR^+$.
Moreover, we consider the leading order source term $\eps^2\,\vf_0$ to be $\scurl$-free, \ie\  
$\scurl\vf_0 = 0$,
having in mind that $\vf_0 = \nabla p_0$ where the pressure $p_0$ corresponds to a solution of a linear and inviscid wave equation.
In addition, we assume for simplicity the source to disappear on the boundary.
The impedance boundary conditions with additional terms due to more general source functions will be given in the Appendix~\ref{sec:appendix:IBC:source}.

For these small acoustic excitations the leading part of the solution satisfies a linear equation in 
the whole domain as considered in~\cite{Schmidt.Thoens.Joly:2014} and the nonlinearity will come 
into play on a higher order. The small viscosities on the other hand leads to boundary layers 
whose thickness becomes proportional to~$\eps$. Indicating their dependency on~$\eps$ we will 
label the acoustic velocity $\vv^\eps$ and the acoustic pressure $p^\eps$ with a 
superscript~$\eps$. 
They are described by the system
\begin{align}
\label{eq:navier.stokes:eps}
  {\mathcal{L}^\eps_k}(\vV^\eps, P^\eps)(\vx) + {\mathcal{N}_k}(\vV^\eps,P^{\eps})(\vx) = 
  \eps^2\, 
  \begin{pmatrix}
     \vf_0(\vx) + \eps\vf_1(\vx) + \eps^2\vf_2(\vx)  \\
     0 
  \end{pmatrix}\delta_{k=1},
\end{align}
{with the vectors $\vV^\eps = \big({\vv}^\eps_0, {\vv}^\eps_1, \dots\big)^\top$,
$P^\eps = \big(p^\eps_0, p^\eps_1, \dots \big)^\top$ of velocity and pressure coefficients}
and {the} linear differential operators 
\begin{align*}
\mathcal{L}^\eps_k({\vV}^\eps, P^\eps) = \left(\begin{aligned}
                   &-\imag k \omega\vv^\eps_k -\eps^2\nu_0 \Delta\vv^\eps_k + \nabla p^\eps_k\\ 
                   &-\imag k \omega p^\eps_k + c^2\Div\vv^\eps_k
                  \end{aligned}\right)\ .
\end{align*}
{In the following we specify first the domain and its boundary before we introduce the ansatz for 
an asymptotic expansion with far field terms and near field correctors and their coupling conditions.}

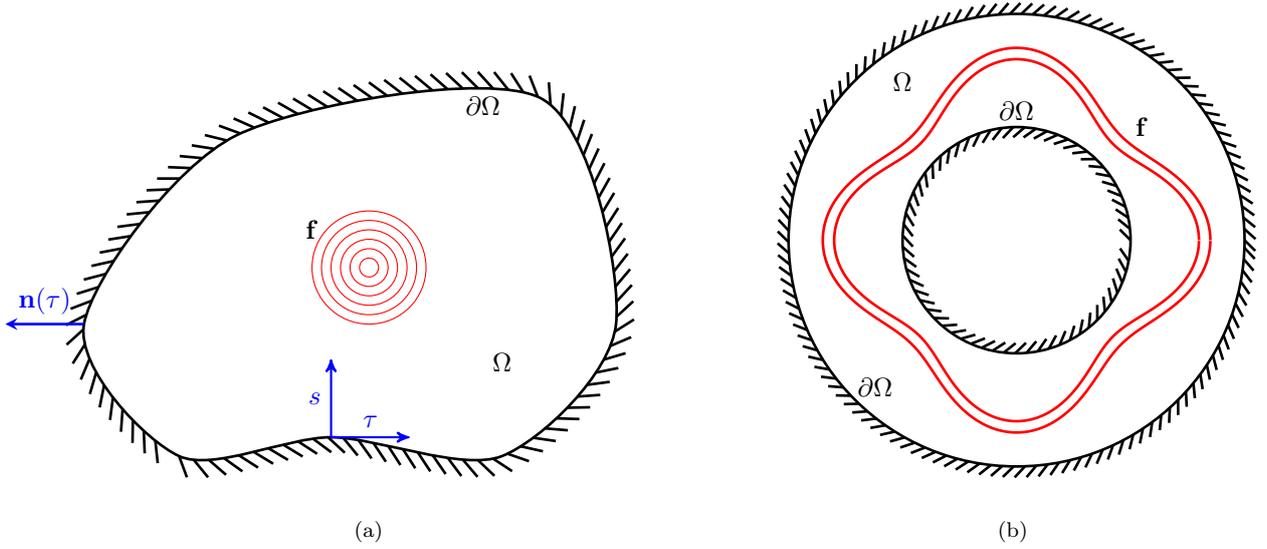
\begin{figure}[hbt]
\centering
\subfigure[]
{
        \label{fig:coordinate}
  \begin{tikzpicture}[scale=2.5,allow upside down,
    interface/.style={
        postaction={draw,decorate,decoration={border,angle=-45,
                    amplitude=0.3cm,segment length=2mm}}},
     ]
    \draw [line width=1pt,interface] plot [smooth cycle] coordinates 
    {(0,0) (0.5,-0.7) (1.3,-0.6) (2.2,-0.7) (2.8,0) (2.4,1.2) (0.8,1)}
    node at (2.2,-0.2) {$\Omega$}
    node at (2.1,1.16) {$\partial \Omega$}
    node at (1.3,-0.6) [sloped,inner sep=0cm,above,anchor=south west,
    minimum height=1cm,minimum width=1cm](N){}
    node at (0,0) [sloped,inner sep=0cm,above,anchor=north east,
        minimum height=1cm,minimum width=1cm](N2){};
    \path (N.south west)%
           edge[-stealth',blue,thick] node[left] {$ s$} (N.north west)
           edge[-stealth',blue,thick] node[above] {$ \td$} (N.south east);
    \path (N2.north west)%
           edge[stealth-,blue,line width=1pt] node[above] {$\vn(\td)$} (N2.north east);
    \foreach \ll in {1, ..., 6} { 
    \draw[color=red] (1.5,0.3) circle (\ll/20);
    };
    \draw (1.2,0.5) node {$\vf$};
\draw [color = white] (3,0) rectangle (3.5,1);
  \end{tikzpicture} 
}
\subfigure[]{
  \label{fig:torus}
  \begin{tikzpicture}[scale=3,allow upside down,
    interface/.style={
        postaction={draw,decorate,decoration={border,angle=45,
                    amplitude=0.2cm,segment length=1.5mm}}},
    interfaceout/.style={
        postaction={draw,decorate,decoration={border,angle=45,
                    amplitude=-0.2cm,segment length=1.5mm}}}
     ]
    \draw[line width=1pt,interfaceout] (0, 0) circle (1);
    \draw[line width=1pt,interface] (0, 0) circle (0.5);
    \draw (-0.5,0.7) node {$\Omega$};
    \draw (0.0,0.57) node {$\partial \Omega$};
    \draw (-0.62,-0.65) node {$\partial \Omega$};
    \draw[line width=1pt, color=red,domain=0:6.28,samples=200,smooth] 
          plot (canvas polar cs:angle=\x r,
            radius= {0.75cm+0.1cm*cos(4*\x r)}); 
    \draw[line width=1pt, color=red,domain=0:6.28,samples=200,smooth] 
          plot (canvas polar cs:angle=\x r,
            radius= {0.7cm+0.1cm*cos(4*\x r)}); 
    \draw (0.55,0.5) node {$\vf$};
  \end{tikzpicture}
}
  \caption{(a) Definition of a general domain with a local coordinate system $(\td,s)$ 
close to the wall; (b) Definition of an annulus domain for numerical simulations.}
\end{figure}

\smallskip
\paragraph
{\em The geometrical setting}
Let $\Omega \subset \IR^2$ be a bounded domain with smooth boundary
$\partial\Omega$. %
The boundary shall be described by a mapping
$\vx_{\partial\Omega}: \td \in \Gammat\to \IR^2$ from a one-dimensional reference domain $\Gammat 
\subset \IR$. %
We assume the boundary to be $C^\infty$ such that 
points {in some neighbourhood $\Omega_\Gamma$ of $\partial\Omega$}
 can be uniquely written as
\begin{align}
  \label{eq:localCoord}
  \vx(\td,s) = \vx_{\partial\Omega}(\td) - s {\vn(\vx_{\partial\Omega}(\td))}
\end{align}
where $\vn$ is the outer normalised normal vector and $s$ the
distance from the boundary. %

Without loss of generality we assume $|\vx_{\partial\Omega}'(\td)| =
1$ for all $\td \in \Gammat$. %
The orthogonal unit vectors in these tangential and normal coordinate
directions are $\ve_\td = -\vn^\bot$, where we use the notation
$\vu^\bot = (u_2, -u_1)^\top$ for a turned vector clockwise by 
$90^\circ$, and $\ve_s = -\vn$. %
{This allows us to write the tangential derivative $\nabla u(\vx)\cdot \ve_\td$ of a function $u \in C^1(\Omega)$ 
with abuse of notation as
\begin{align}
  \label{eq:tangentialderivative}
  (\partial_\td u)(\vx) := \partial_\tau u(\vx_{\partial\Omega}(\td)).
\end{align}
}
Moreover, the curvature {$\kappa$} on the boundary $\partial\Omega$ is given by
\begin{align*}
  \kappa(\vx_{\partial\Omega}(\td)) := 
  \frac{x_{\partial\Omega,1}'(\td)x_{\partial\Omega,2}''(\td) - x_{\partial\Omega,2}'(\td)x_{\partial\Omega,2}''(\td)}{(x_{\partial\Omega,1}'(\td)^2+x_{\partial\Omega,2}'(\td)^2)^{3/2}}\ .
\end{align*}

\paragraph
{\em Asymptotic ansatz.}
Within this article we consider the acoustic source of the same order as the boundary layer,
\ie\ $\vf = \vf^\eps = O(\eps^2)$. {In the linear model the resulting acoustic velocity 
and pressure are of the same order and for the considered nonlinear model the same is true.}
The solution {$\vV^\eps$, $P^\eps$ of \eqref{eq:navier.stokes:eps}} should be approximated by a
two-scale asymptotic expansion in the framework of Vishik and
Lyusternik~\cite{Vishik.Lyusternik:1960} {and for each coefficient we take the ansatz} 
\begin{align}
  \label{eq:asympExpan}
  \vv_k^\eps \sim 
  \sum_{j=0}^{\infty} \eps^{j+2}\left(\vv_k^j (\vx) +
    {\vv^j_{\BL,k}}(\vx)\right), \qquad 
  p_k^\eps \sim 
  \sum_{j=0}^{\infty} \eps^{j+2}\left(p_k^j(\vx) + {p^j_{\BL,k}}(\vx)\right)\ ,
  && {\text{for }\eps\to 0}\ ,
\end{align}  
where $\vv^j_k$ and $p^j_k$ are the \textit{far field} {velocity and pressure of order~$j$} 
and $\vv^j_{\BL,k}$ and $p^j_{\BL,k}$ represent {the respective} \textit{near field} velocity and 
pressure. {They are seeked in scaled coordinate $S(s)=\tfrac{s}{\eps}$ of the} local normalised 
coordinate system~\eqref{eq:localCoord} in the form
\begin{subequations} 
\label{eq:BL:expan}
  \begin{align} 
    {\vv^j_{\BL,k}}(\vx) &= \Phi^j_{k,\td} (\td,\tfrac{s}{\eps})\ \ve_\td(\td) +
    {\Phi^j_{k,s}}(\td,\tfrac{s}{\eps})\ \ve_s(\td) \\
    {p^j_{\BL,k}}(\vx) &= \Pi^j_k(\td,\tfrac{s}{\eps})
  \end{align} 
\end{subequations}
{taking into account the fact that} the boundary layer thickness scales {linearly} with 
$\eps$. For the desired decay properties we require the near field terms 
$\Phi^j_{k,\td}(\tau,S)$, $\Phi^j_{k,s}(\tau,S)$ and $\Pi^j_k(\tau,S)$ as well as 
their higher derivatives to vanish with $S \to \infty$.
The subscript $\cdot_\BL$ stands for ``boundary layer'' expressing
the nature of the near field terms that they are essentially defined in a small layer close the 
boundary.
Indeed the equality~\eqref{eq:BL:expan} can be assumed to be true only in an $O(1)$ neighbourhood 
of the boundary in which the local coordinate system~\eqref{eq:localCoord} is defined. Outside this 
neighbourhood the expression on the right hand sides of~\eqref{eq:BL:expan}, that decaying 
exponentially in an $O(\eps)$ distance, is multiplied with a smooth cut-off function such that the 
product is exactly zero where the local coordinate system is not defined.


In the linear case~\cite{Schmidt.Thoens.Joly:2014} the near field velocity turned out to be 
divergence free such that there is no boundary layer for the pressure. Due to the coupling of the 
velocity and pressure by the nonlinear terms {this property can not be assumed in general. 
However, we will see in our analysis that the near field} pressure terms $\Pi^j_{\BL,k}$ vanish at 
least up to order 2 and {up to this order} the resulting near field terms for the frequency 
$\omega$ of the excitation remain exactly the same as for the linear system. 


\smallskip
\paragraph{\em Coupling of far and near field by the no-slip boundary conditions}
By the homogeneous Dirichlet boundary condition %
the tangential trace {$\big(\vv^{\eps}_k + \vv^{\BL}_k\big)\cdot\ve_\td$} and %
    normal     trace {$\big(\vv^{\eps}_k + \vv^{\BL}_k\big)\cdot\vn$} %
vanish for any $k$ and {separately in the orders} in $\eps$, 
cf~\eqref{eq:navier.stokes:orig:bound}, therefore the traces of the far field have to 
fulfil the conditions
\begin{subequations}
\label{eq:coupling}
\begin{align}
  {\Phi^j_{k,\td}}(\td,0) = -{v^j_{k,\td}}(\td) &:= -\vv^{j}_k(\vx(\td,0))\cdot\ve_\td\ , 
\\
  {\Phi^{j}_{k,s}}(\td,0) = 
-{v^{j}_{k,s}}(\td) &:= \hspace{0.8em}\vv^{j}_k({\vx(\td,0)})\cdot\vn\ .
\end{align} 
\end{subequations}

The far and near field terms will be derived order by order up order~2 in Section~\ref{sec:derivation} %
as well as the effective systems with impedance boundary conditions that we will present already in the %
following subsection.



\subsection{Effective systems with impedance boundary conditions}
\label{sec:ModelDef}

In this section we present effective models of order {$N = 0$, $1$ and $2$ 
for approximative far field solutions %
{$\vV^{\eps,N} = (\vv^{\eps,N}_0, \vv^{\eps,N}_1, \ldots)^\top$, %
$P^{\eps,N} = (p^{\eps,N}_0, p^{\eps,N}_1, \ldots)^\top$}, 
in which the nonlinear and viscous behaviour in the layers close to the boundary
are incorporated with 
impedance boundary conditions. 
The steps for deriving the systems for approximative velocity and pressure
will follow in Section~\ref{sec:derivation:effectiveSystems}. 
Contrary to the original system~\eqref{eq:navier.stokes:orig} in time domain
or its multiharmonic approximation~\eqref{eq:navier.stokes:eps}, for which all modes
of velocity and pressure couple, the approximative pressure and velocity coefficients
decouple for all modes $k > 0$. Therefore, we introduce separately systems for 
pressure coefficients $p^{\eps,N}_k$ only, where associated velocity coefficients $\vw^{\eps,N}_k$
are defined afterwords as a function of the pressure, and systems for velocity coefficients 
$\vv^{\eps,N}_k$ only, where associated pressure coefficients $q^{\eps,N}_k$ follow directly. %
Only for the static mode $k = 0$ we have coupled velocity and pressure systems. %
In general, the directly defined pressure coefficients $p^{\eps,N}_k$ and the pressure coefficients
$q^{\eps,N}_k$ computed from the velocity may differ as well as the two velocity approximations 
$\vv^{\eps,N}_k$ and $\vw^{\eps,N}_k$. We also distinguish the two for the static mode $k = 0$
even so here velocity and pressure coefficients are defined in a coupled system as the right 
hand side of this system depends on the different approximations.

Derived from the asymptotic expansion, both approximative far field solutions for velocity and 
pressure order $N$ shall be close to the respective far field expansion of order $N$, \ie\ for the} 
frequency mode $k$ {we expect that} 
\begin{align}
 \label{eq:appr:form}
  \begin{pmatrix}
    \vv^{\eps,N}_k\\
    q^{\eps,N}_k
  \end{pmatrix},
  \begin{pmatrix}
    \vw^{\eps,N}_k\\
    p^{\eps,N}_k
  \end{pmatrix}
   = \sum_{j=0}^N \eps^{j+2}
   \begin{pmatrix}
     \vv^j_k \\
     p^j_k
    \end{pmatrix} + O(\eps^{N+3})\ . 
\end{align}

Even equally important the asymptotic regime of small sound amplitudes leads to an iterative procedure
to obtain the coefficients for different modes $k$ (see Table~\ref{tab:approximateModels}).
In general, the coefficients for $k = 1$ can be defined independently and the neighbouring modes for $k = 0$
and $k = 2$ follow. Moreover, up to order $2$ there are no modes for $k > 2$ as 
indicating that the response at the higher harmonics $3\,\omega, 4\,\omega, \ldots$
are more than two orders in $\eps$ smaller than the excitation amplitude. %
In general, for approximation of order $N$ we have only the modes $k = 0, 1,\ldots, \lceil \frac{N+1}{2}\rceil$.

\begin{table}
  \centering
  \begin{tabular}{c@{\hspace{3em}}ccccc c@{\hspace{4em}}ccccc}
    order & \multicolumn{5}{c}{pressure} && \multicolumn{5}{c}{velocity}\\
    \hline\\[-0.8em]
    $N = 0$ & & & $p^{\eps,0}_1$ & & & & $\vv^{\eps,0}_0$ & $\Longleftarrow$ & $\vv^{\eps,0}_1$ \\
    $N = 1$ & & & $p^{\eps,1}_1$ & & & & $\vv^{\eps,1}_0$ & $\Longleftarrow$ & $\vv^{\eps,1}_1$ \\
    $N = 2$ & $p^{\eps,2}_0$ & $\Longleftarrow$ & $p^{\eps,2}_1$ & $\Longrightarrow$ & 
$p^{\eps,2}_0$ &
            & $\vv^{\eps,2}_0$ & $\Longleftarrow$ & $\vv^{\eps,2}_1$ & $\Longrightarrow$ & $\vv^{\eps,2}_2$\\[0.4em]
    \hline\\[-0.8em]
  \end{tabular}
  \caption{Pressure approximations $p^{\eps,N}_k$ and velocity approximations $\vv^{\eps,N}_k$ 
           can be computed separately from each other (except for static mode for $k = 0$) and
           sequentially in the mode index~$k$. Then, velocity approximations
           $\vw^{\eps,N}$ are directly deduced from $p^{\eps,N}_k$ and pressure approximations
           $q^{\eps,N}$ from $\vv^{\eps,N}$. The velocity and pressure approximations for 
           the excitation frequency ($k = 1$) coincides with the respective solutions 
           for the linear case up to order~1 (see~\cite{Schmidt.Thoens:2014}).}
  \label{tab:approximateModels}
\end{table}

\subsubsection{Systems for the pressure}
\label{sec:ModelDef:pressure}

{Here we present the approximative models for the far field
pressure. This is different to the original equations in which no boundary
conditions for the pressure, but for both velocity components, are
imposed for each order. Approximative velocities can be computed {\em a-posteriori}
{(see Sec.~\ref{sec:ModelDef:pressure:velocityPostProc})}. %

\medskip
\paragraph{\bf\em Order $O(\eps^2)$}
The approximative model for the pressure in frequency of the 
excitation $1\cdot\omega$ is given by a linear system
\begin{subequations}
\label{eq:pappr:0}
\begin{align}
  \label{eq:pappr:01:PDE}
  \Delta {p}^{\eps,0}_1 + \frac{\omega^2}{c^2} {p}^{\eps,0}_1 &= \Div \vf,
    && \text{in }\Omega,  \\
  \label{eq:pappr:01:bc}
  \nabla {p}^{\eps,0}_1 \cdot \vn  &= 0,    && \text{on }\partial\Omega,
\end{align}
\end{subequations}
{All the terms $p_k^{\eps,0}$ for $k \neq 1$ are 
zero, meaning that the limit acoustic pressure is exactly as in the 
linear case.}

\medskip
\paragraph{\bf\em Order $O(\eps^3)$}
The approximative model in frequency $1\cdot\omega$ is given by the linear system
\vspace{-0.3em}
\begin{subequations}
\label{eq:pappr:1}
\begin{align}
  \label{eq:pappr:11:PDE}
  \Delta {p}^{\eps,1}_1 + \frac{\omega^2}{c^2} {p}^{\eps,1}_1 &= \Div \vf,
    && \text{in }\Omega,  \\
  \label{eq:pappr:11:bc}
  \nabla {p}^{\eps,1}_1 \cdot \vn + (1+\imag)\sqrt{\frac{\nu}{2\omega}}\partial_\td^2 
  {p}^{\eps,1}_1  &= 0,    && \text{on }\partial\Omega.
\end{align}
\end{subequations}
Again, all the terms $p_k^{\eps,1}$ for $k \neq 1$ are 
zero and the resulting acoustic pressure approximation {is} exactly as in the 
linear case. 
The impedance boundary conditions~\eqref{eq:pappr:11:bc}
are of Wentzell type. %
See \cite{BonnaillieNoel.Dambrine.Herau.Vial:2010,Schmidt.Heier:2015}
for the functional framework and variational formulation. %

\medskip
\paragraph{\bf\em Order $O(\eps^4)$}
For frequency $1\cdot\omega$ the pressure of order 2 is solution of 
\begin{subequations}
  \label{eq:pappr:2}
  \begin{align}
    \label{eq:pappr:12:PDE}
    \Big(1- \frac{\imag\omega\nu}{c^2}\Big)\Delta {p}^{\eps,2}_1 +
    \frac{\omega^2}{c^2}{p}^{\eps,2}_1 &= \Div \vf,
    && \text{in }\Omega,  \\
    \label{eq:pappr:12:bc}
    \nabla {p}^{\eps,2}_1 \cdot \vn + (1+\imag)\sqrt{\frac{\nu}{2\omega}}
    \partial_\td^2 {p}^{\eps,2}_1 
    +\frac{\imag\nu}{2\omega}\partial_\td(\kappa\partial_\td{p}^{\eps,2}_1) &= 0,
    && \text{on }\partial\Omega.
  \end{align}
\end{subequations}
Even for $N=2$ the nonlinear terms do not affect the {pressure} approximation in 
frequency of excitation, which {coincides with the approximations in} the linear case
{and are also obtained via a system decoupled from}
the velocity. 
However, in this order {of approximation the first time}
other frequency modes come into play, namely that for the}
frequency $0 \cdot\omega$, a so called acoustic 
streaming~\cite{Lighthill:1978}, {and for the frequency $2\cdot\omega$.}
The acoustic pressure {at frequency $0 \cdot\omega$} is explicitly defined by {the algebraic equation}
\begin{align}
    \label{eq:pappr:02}
 {p}^{\eps,2}_0 = - \frac{1}{4\omega^2} \big|\vf - \nabla {p}^{\eps,2}_1\big|^2
\end{align}
and {the one at} frequency $2\cdot\omega$ by {the Helmholtz equation}
\begin{subequations}
  \label{eq:pappr:22}
  \begin{align}
    \label{eq:pappr:22:PDE}
    \Delta {p}^{\eps,2}_2 + \frac{4\,\omega^2}{c^2} {p}^{\eps,2}_2  &= 
  {\frac{1}{4\omega^2} \Delta \big(\vf - \nabla p^{\eps,2}_1 \big)^2 }
  + \frac{1}{c^2} \left( \big(\vf - \nabla p^{\eps,2}_1\big) \cdot \nabla 
  p^{\eps,2}_1 {+} \frac{\omega^2}{c^2} \big(p^{\eps,2}_1\big)^2 \right),
    && \text{in }\Omega, \\
    \label{eq:pappr:22:bc}
    \nabla{p}^{\eps,2}_2\cdot\vn &= 0,    && \text{on }\partial\Omega. 
  \end{align}
\end{subequations}
For a well-posed definition of $p^{\eps,2}_2$
the source function $\vf$ has to be continously differentiable 
and also the pressure approximation $p^{\eps,2}_1$ needs higher regularity. 
{Note, that the right hand side of~\eqref{eq:pappr:22:PDE} can be simplified using the fact that $p^{\eps,2}_1$
is solution to an Helmholtz problem (see Appendix~\ref{sec:appendix:p22}).}
For a numerical approximation with $C^0$-continuous finite elements, for which this regularity 
is only attained approximately, the right hand side can be evaluated 
as the projection of the pressure gradient $\nabla p^{\eps,2}_1$ to continuous vector fields.

\subsubsection{Post-processing of velocity from systems for the pressure}
\label{sec:ModelDef:pressure:velocityPostProc}

When the far field pressure is computed we may obtain {\em a-posteriori} 
{approximations to} the far field velocity {to the respective} order. 

{The far field velocities at frequency $1\cdot\omega$ are defined at the different approximation orders by}
\begin{align}
\label{eq:wappr:k=1}
 {\vw}^{\eps,0}_1 &= \frac{\imag}{\omega}\left(\vf - \nabla {p}^{\eps,0}_1\right),
 &
 {\vw}^{\eps,1}_1 &= \frac{\imag}{\omega}\left(\vf - \nabla {p}^{\eps,1}_1\right),
 & 
  {\vw}^{\eps,2}_1 &= \frac{\imag}{\omega}\left(\vf - \nabla {p}^{\eps,2}_1\right)
  - \frac{\nu}{c^2}\nabla {p}^{\eps,2}_1.
\end{align}
{and those at frequency $2\cdot\omega$ at order $2$ by}
\begin{align}  
  \label{eq:wappr:k=2}
  {\vw}^{\eps,2}_2 &= -\frac{\imag}{2\,\omega}\left(
    \nabla {p}^{\eps,2}_2 + 
    \frac{1}{4\,\omega}\nabla \big| \vf - \nabla {p}^{\eps,2}_1\big|^2\right). 
\end{align}

For the frequency $0\cdot\omega$ a far field velocity approximation $\vw^{\eps,1}_0$ of order~1
can be obtained as solution of linear Stokes system similarly to~\eqref{eq:vappr:1:k=0}
in the following subsection that is directly for a velocity approximation, however, using $\vw^{\eps,1}_1$ on its right hand side. 
Likewise, a far field velocity approximation $\vw^{\eps,2}_0$ for order~2 can be defined
by a nonlinear Navier-Stokes like system as~\eqref{eq:vappr:2:k=0} that depends on $\vw^{\eps,2}_1$.

The far field velocity can be used as approximation away from the
boundary and has to be corrected by a near field velocity {approximation (see Sec.~\ref{sec:ModelDef:nearfieldvelocity})}.

\subsubsection{Systems for the velocity}
\label{sec:ModelDef:velocity}

Here we propose approximative models {directly} for the far field velocity.
For each order an approximative pressure can be computed afterwards {(see Sec.~\ref{sec:ModelDef:velocity:pressurePostProc})
as well as a near field velocity approximation (see Sec.~\ref{sec:ModelDef:nearfieldvelocity})}.

\medskip
\paragraph{\bf\em Order $O(\eps^2)$}
The {limit} model is given by a linear system in frequency of excitation
\vspace{-0.3em}
\begin{subequations}
 \label{eq:vappr:0}
  \begin{align}
    \label{eq:vappr:10:PDE}
    \nabla\Div {\vv}^{\eps,0}_1 + \frac{\omega^2}{c^2} {\vv}^{\eps,0}_1 &= 
    \frac{\imag\omega}{c^2} \vf, && \text{ in }\Omega,\\[0.2em]
    \label{eq:vappr:10:bc}
    {\vv}^{\eps,0}_1\cdot\vn &= 0,&& \text{ on }\partial\Omega,
  \end{align}
\end{subequations}%
and all other terms $\vv_k^{\eps,0}$, $k \neq 1$ are 
zero. So the limit acoustic velocity coincides with the one in the 
linear case.

\medskip
\paragraph{\bf\em Order $O(\eps^3)$}
In frequency of excitation the approximative model is given by
\vspace{-0.5em}
\begin{subequations}
 \label{eq:vappr:1}
  \begin{align}
    \label{eq:vappr:11:PDE}
    \nabla\Div {\vv}^{\eps,1}_1 + \frac{\omega^2}{c^2} {\vv}^{\eps,1}_1 &= 
    \frac{\imag\omega}{c^2} \vf, && \text{ in }\Omega,\\[0.2em]
    \label{eq:vappr:11:bc}
    {\vv}^{\eps,1}_1\cdot\vn - 
    (1+\imag)\frac{c^2}{\omega^2}
    \sqrt{\frac{\nu}{2\omega}}
    \partial_\td^2\Div{\vv}^{\eps,1}_1%
    &= 0, && \text{ on }\partial\Omega,
  \end{align}
\end{subequations}
 and {there is a non-zero acoustic streaming velocity at} frequency $0\cdot\omega$ {that satisfies
the Stokes system} 
\begin{subequations}
  \label{eq:vappr:1:k=0}
  \begin{align}
    \label{eq:vappr:01:PDE}
    - \nu \Delta {{\vv}^{\eps,1}_0}  + \nabla q^{\eps,3}_0  &= 
  -{\frac{1}{4}} \left( ({\vv^{\eps,1}_1}\cdot\nabla)\overline{\vv^{\eps,1}_1}
  + (\overline{\vv^{\eps,1}_1}\cdot\nabla) {\vv^{\eps,1}_1}\right),
     && \text{in }\Omega\ ,\\
    \Div {\vv}^{\eps,1}_0 &= 0,  && \text{in }\Omega\ , \\
    \label{eq:vappr:01:bc}
    \vv^{\eps,1}_0 &= \zerobf, && \text{on } \partial\Omega\ .
    \end{align}
  \end{subequations}
{The purely real right hand side of \eqref{eq:vappr:1:k=0} implies that 
its solution $({\vv}^{\eps,1}_0, q^{\eps,3}_0)$ is purely real.
  Note that} $q^{\eps,3}_0$ is not only a Lagrange multiplier but a higher order approximation of the pressure at zero frequency.

\medskip
\paragraph{\bf\em Order $O(\eps^4)$} The approximative model in frequency $1\cdot\omega$ is defined 
by
\begin{subequations}
 \label{eq:vappr:2}
 \begin{align}
   \label{eq:vappr:12:PDE}
    \left(1-\frac{\imag\omega\nu}{c^2}\right)\nabla \Div {\vv}^{\eps,2}_1 + 
    \frac{\omega^2}{c^2} {\vv}^{\eps,2}_1 
    &= \frac{\imag\omega}{c^2} \vf, 
    && \text{ in }\Omega,\\[0.2em]
    \label{eq:vappr:12:bc}
    {\vv}^{\eps,2}_1\cdot\vn - 
    \frac{c^2}{\omega^2}\Big(
    (1+\imag)\sqrt{\frac{\nu}{2\omega}}
    \partial_\td^2\Div{\vv}^{\eps,2}_1%
    &+ \frac{\imag\nu}{2\omega}
    \partial_\td(\kappa\partial_\td\Div {\vv}^{\eps,2}_1)\Big) 
    = 0, && \text{ on }\partial\Omega, 
\end{align}
\end{subequations}
\begin{subequations}
  \label{eq:vappr:2:k=0}
{and} that of frequency $0\cdot\omega$ by {the nonlinear system}
  \begin{align}
    \label{eq:vappr:02:PDE}
    - \nu \Delta {\vv}^{\eps,2}_0 &+ {
  ({\vv^{\eps,2}_0}\cdot\nabla)\vv^{\eps,2}_0} 
    + \nabla q^{\eps,4}_0 = -{\frac{1}{4}\big(
    ({\vv^{\eps,2}_1}\cdot\nabla)\overline{\vv^{\eps,2}_1}+
    (\overline{\vv^{\eps,2}_1}\cdot\nabla){\vv^{\eps,2}_1} \Big)},
     && \text{in }\Omega\\
    \Div {\vv}^{\eps,2}_0 &=  
    {-\frac{1}{4c^2} \big( {\vv}^{\eps,2}_1\cdot \overline{\vf} + \overline{\vv}^{\eps,2}_1\cdot {\vf} \big)},
    && \text{in }\Omega \\
    %
    \label{eq:vappr:02:bc}
    \vv^{\eps,2}_0 &= \zerobf, && \text{on } \partial\Omega
    \end{align}
\end{subequations}
Again, the solution {$({\vv}^{\eps,2}_0,q^{\eps,4}_0)$} of~\eqref{eq:vappr:2:k=0} is purely real
  and $q^{\eps,4}_0$ is a pressure approximation of higher order, where $q_0^{\eps,4} = -\frac{1}{4} \big| {\vv}^{\eps,2}_1 \big|^2 + O(\eps^5)$.
At frequency $2\cdot\omega$ a velocity approximation satisfies the Helmholtz equation
\begin{subequations}
  \label{eq:vappr:22}
  \begin{align}
    \label{eq:vappr:22:PDE}
    \nabla\Div {\vv}^{\eps,2}_2 + \frac{4\omega^2}{c^2} {\vv}^{\eps,2}_2 = 
    &     {- \frac{\imag\omega}{c^2} \nabla \big(\vv^{\eps,2}_1\big)^2}
    - \frac{1}{2c^2} \nabla (\vv^{\eps,2}_1 \cdot \vf)
    + \frac{\imag}{2\omega} \nabla 
    \big(\Div {\vv}^{\eps,2}_1 \big)^2,
    && \text{ in }\Omega\ ,\\[0.2em]
    \label{eq:vappr:22:bc}
    {\vv}^{\eps,2}_2\cdot\vn &= 0, && \text{ on }\partial\Omega\ .
  \end{align}
\end{subequations}

\begin{remark}
 Note, that $({\vv^{\eps,2}_0}\cdot\nabla)\vv^{\eps,2}_0 = 
({\vv^{\eps,1}_0}\cdot\nabla)\vv^{\eps,1}_0  + O(\eps^5)$. Therefor instead of solving the 
nonlinear system~\eqref{eq:vappr:02:PDE}--\eqref{eq:vappr:02:bc} one could, first, find 
${\vv^{\eps,1}_0}$ by solving the linear system~\eqref{eq:vappr:01:PDE}--\eqref{eq:vappr:01:bc}, 
and 
then substitute it in~\eqref{eq:vappr:02:PDE}--\eqref{eq:vappr:02:bc} for ${\vv^{\eps,2}_0}$ in 
the advection term. That will lead again to a linear system.
\end{remark}


\subsubsection{Post-processing of pressure from systems for the velocity}
\label{sec:ModelDef:velocity:pressurePostProc}
When the far field  velocity {approximation} is computed we may obtain {\em a-posteriori} 
{an associated} far field  pressure approximation for the frequencies $1\cdot\omega$ and $2\cdot\omega$.
The approximations for frequency $1\cdot\omega$ are given by
\begin{align}
 \label{eq:pappr:N1}
 {q}^{\eps,N}_1 &= -\frac{\imag c^2}{\omega}\Div {\vv}^{\eps,N}_1, \quad {N=0,1,2}\ , 
\end{align}
{and the approximation of order 2 for frequency $2\cdot\omega$ by}
\begin{align}
 \label{eq:qappr:22}
 {q}^{\eps,2}_2 &= -\frac{\imag c^2}{2\omega}\Div {\vv}^{\eps,2}_2
  -\frac{\imag}{{2}\omega} {\vv}^{\eps,2}_1\cdot \vf 
    {-} \frac{c^2}{{2}\omega^2}  
    \left( \big(\Div {\vv}^{\eps,2}_1 \big)^2 - 
    \frac{\omega^2}{c^2}\big({\vv}^{\eps,2}_1 \big)^2 \right)\ .
\end{align}
Moreover, a pressure approximation of order 2 at frequency $0\cdot\omega$ is given by 
\begin{align}
  {q}^{\eps,2}_0 &= -\frac{1}{4} \big| {\vv}^{\eps,0}_1 \big|^2\ .
\end{align}

\subsubsection{Post-processing of a near field velocity} 
\label{sec:ModelDef:nearfieldvelocity}

Close to the wall the far field velocity {approximations $\vV^{\eps,N}$
have to be corrected by boundary layer} functions in tangential {as well as normal} direction 
\begin{align}
 \label{eq:vappr:BLtau}
  {\vV}^{\BL,N}(\vx) = \chi(\vx) \sum_{\ell=0}^N
  \vH^{\ell}(V^{\eps,N}_\tau)(\vx)\ \mathrm{e}^{-(1-\imag)\sqrt{\frac{\omega}{2\nu}}\, s(\vx)},  \quad
\end{align}
where $\chi$ is an admissible cut-off function (see~\cite{Schmidt.Thoens.Joly:2014})
that takes the constant value $1$ in some subset of $\Omega_\Gamma$, %
$s$ is the distance function to the boundary, \ie\ there exists 
for each point $\vx \in \Omega_\Gamma$ a base point $\vx_{\partial\Omega} \in \partial\Omega$ such that 
$\vx = \vx_{\partial\Omega} + s(\vx)\vn^\bot(\vx_{\partial\Omega})$,
and the operators $\vH^{\ell}: (C^\infty(\partial\Omega))^\infty \to (C^\infty(\Omega_\Gamma))^\infty$ %
with $\vH^{\ell} = (0, \vh^{\ell}_1,\vh^{\ell}_2,\ldots)^\top$
that are acting on the tangential velocity traces $V^{\eps,N}_\tau = (v^{\eps,N}_{0,\tau}, v^{\eps,N}_{1,\tau}, \ldots)^\top$
with $v^{\eps,N}_{k,\tau}(\vx_{\partial\Omega}) := \vv^{\eps,N}_k(\vx_{\partial\Omega})\cdot \vn^\bot(\vx_{\partial\Omega})$
where we note that $v^{\eps,N}_{k,\tau} = 0$ for $k > \lceil \frac{N+1}{2}\rceil$.
Hence, to define $\vV^{\BL,N}$ for $N = 0,1,2$ we state the operators
\begin{subequations}
  \begin{align}
     \vh^{0}_{1}(V^{\eps,N}_\tau)(\vx) &= %
     - v^{\eps,N}_{1,\tau} \vn^\bot\ ,\\[0.2em]
     \vh^{1}_{1}(V^{\eps,N}_\tau)(\vx) &= %
     -\tfrac12 \kappa s(\vx) v^{\eps,N}_{1,\tau} \vn^\bot %
     + (1+\imag) \sqrt{\tfrac{\nu}{2\omega}}\partial_\td v^{\eps,N}_{1,\td}\vn\ ,\\[0.2em]
     \nonumber
     \vh^{2}_{1}(V^{\eps,N}_\tau)(\vx) &= %
     -\tfrac{3}{8} \kappa^2 s(\vx)\, \Big((1+\imag) \sqrt{\tfrac{2\nu}{\omega}}- s(\vx)\Big) 
      v \vn^\bot - (1+\imag)\sqrt{\tfrac{\nu}{2\omega}}\, s(\vx)\, \partial^2_\td v \vn^\bot\\[-0.2em]
     &\hspace{1em} 
     + \tfrac12(1+\imag) \sqrt{\tfrac{\nu}{2\omega}}\left(2\kappa \partial_\td v^{\eps,N}_{1,\tau} + \partial_\td\kappa v^{\eps,N}_{1,\tau}\right)s(\vx) \vn
     + \tfrac{\imag \nu}{2\omega} \partial_\td\kappa v^{\eps,N}_{1,\tau} \vn\ ,
     \\
     \vh^{2}_{2}(V^{\eps,N}_\tau)(\vx) &= - v^{\eps,N}_{2,\tau} \vn^\bot\ .
  \end{align}
\end{subequations}
where we note that $\kappa$, $\vn^\bot$, $\vn$ and $v^{\eps,N}_{1,\tau}$ are functions of the base point $\vx_{\partial\Omega}$ of $\vx$
and $\partial_\td$ is the tangential derivative defined in~\eqref{eq:tangentialderivative}.
\section{Derivation of terms of multiscale expansion and effective systems}
\label{sec:derivation}

In {Sec.~\ref{sec:Asymptotic}} we have introduced the {ansatz of the} two-scale expansion~\eqref{eq:asympExpan}, which expresses an approximation to the exact solution as a 
two-scale \textit{decomposition} into far field terms, modelling the 
macroscopic picture of the solution, which are corrected in the 
neighbourhood of the boundary by near field terms. To separate the two 
scales we use the technique of multiscale expansion as described 
in Sec.~\ref{sec:Asymptotic}, which defines the near field terms in a local 
normalised coordinate system~\eqref{eq:localCoord} such that they 
decay {rapidly} away from the wall and are set to zero where the local 
coordinate system is not defined (using a cut-off function). 
In the following we define the terms of asymptotic 
expansion~\eqref{eq:asympExpan} order by order.

\subsection{Correcting near field}
\label{sec:NearField}
In this section we will give the near field equations and their solutions up to order 2.
{They are derived} such that the near field velocity and 
pressure expansions~\eqref{eq:BL:expan} inserted 
into~\eqref{eq:navier.stokes:eps} leave a residual as small as possible in 
powers of $\eps$ {and that the sum of tangential far and near field velocity 
vanishes} at the boundary.
The general form of the near field equations of any order and in any frequency can be found in the 
Appendix~\ref{sec:appendix:nearfield}. 

\medskip
\paragraph{\bf\em The near field terms of order $O(\eps^2)$.} 
The near field equation for $j=0$ in frequency $1\cdot\omega$ yields
\begin{align*}
 \imag \omega{u}^0_{1,\td} +\nu_0\de{S}^2 {u}^0_{1,\td} &= 0, \\
  \de{S} {u^{0}_{1,s}} &= 0, \\
  \de{S}q^{0}_1 &= 0.
\end{align*}
It is easy to see that its unique solution together with the coupling condition for far and near 
fields~\eqref{eq:coupling} and decay condition for the near field is given by
\begin{subequations}
  \label{eq:nearField:k1:0}
\begin{align}
  u^0_{1,\td}(\td,S) &= -v^0_{1,\td}(\td)\,\mathrm{e}^{-\lambda_0 S}, \quad
  \text{with }  \lambda_0 = (1-\imag)\sqrt{\omega/2\nu_0}, \\
  u^{0}_{1,s}(\td,S) &= 0, \\
  q^{0}_1(\td,S) &= 0.
\end{align}
\end{subequations}
This is the dominating boundary layer term close to the wall. 

\medskip
\paragraph{\bf\em The near field terms of order $O(\eps^3)$.} 
The near field equations for $j=1$ in frequency $1\cdot\omega$ are given by
\begin{align*} 
  \imag \omega{u}^1_{1,\td} +\nu_0\de{S}^2 {u}^1_{1,\td}
   &= \kappa\big(3\,\imag \omega S + 3\nu_0 S \de{S}^2 + 
  \nu_0\de{S}\big) {u}^0_{1,\td}, \\
  \de{S} {u^{1}_{1,s}} &= -\de{\td} u^{0}_{1,\td}, \\
  \de{S}q^{1}_1 &= 0,
\end{align*} 
which unique solution, using the terms in~\eqref{eq:nearField:k1:0} together 
with the coupling condition, is
\begin{subequations}
\label{eq:nearField:k1:1}
\begin{align}
  u^1_{1,\td}(\td,S) &= -\left(v^1_{1,\td}(\td) +\tfrac{1}{2}\kappa S\, 
v^0_{1,\td}(\td)\right)\,\mathrm{e}^{-\lambda_0 S}, \\
  u^1_{1,s} (\td,S) &= -\frac{1}{\lambda_0}\,\partial_\td 
    v^0_{1,\td}(\td)\,\mathrm{e}^{-\lambda_0 S}, \\
  q^1_{1} (\td,S) &= 0\ .
\end{align}
\end{subequations}

\paragraph{\bf\em The near field terms of order $O(\eps^4)$.} 
The near field equations for $j=2$ in frequency $1\cdot\omega$ are given by
\begin{align*} 
  \imag \omega{u}^2_{1,\td} +\nu_0\de{S}^2 {u}^2_{1,\td}
   &= \kappa\big(3\,\imag	 \omega S + 3\nu_0 S \de{S}^2 + 
  \nu_0\de{S}\big) {u}^1_{1,\td} \\ 
  &-\nu_0\de{\td}^2u^0_{1,\td} -\kappa^2 \big(3\,\imag \omega S^2 + 
    3\nu_0 S^2\de{S}^2  + \nu_0 (2 S\de{S} -1) \big) u^0_{1,\td}, \\
  \de{S} {u^{2}_{1,s}} &= -\de{\td} u^{1}_{1,\td} + 
  \kappa (S\de{S} {u^{1}_{1,s}} + {u^{1}_{1,s}}), \\
  \de{S}q^{2}_1 &= 0,
\end{align*} 
which unique solution, using the terms in~\eqref{eq:nearField:k1:0} 
and~\eqref{eq:nearField:k1:0} together with the coupling condition, is
\begin{align}  
\label{eq:nearField:k1:2}
  u^2_{1,\td}(\td,S) &= -\left(v^2_{1,\td}(\td) + \tfrac12\kappa S v^1_{1,\td}(\td) 
   -\frac{3\kappa^2 S}{8}\, \left(\frac{1}{\lambda_0}-S\right) v^0_{1,\td}(\td) + 
\frac{S}{2\lambda_0} \partial^2_\td v^0_{1,\td}(\td)
  \right)\,\mathrm{e}^{-\lambda_0 S}, \\
  u^2_{1,s}   (\td,S) &= -\frac{1}{\lambda_0} \left(\de{\td}v^1_{1,\td}(\td) + 
\frac{\kappa}{2}
   \Big(3S+\frac{1}{\lambda_0}\Big)\de{\td}v_{1,\td}^0(\td)
     + \frac{\kappa'}{2}
   \Big(S+\frac{1}{\lambda_0}\Big)v_{1,\td}^0(\td) \right)\,\mathrm{e}^{-\lambda_0 S}, \\
  q^2_1 &= 0.
\end{align}

In frequency $2\cdot\omega$ the first non trivial terms appear for $j=2$ with 
the near field equations given by
\begin{align*}
 2\imag \omega{u}^2_{2,\td} +\nu_0\de{S}^2 {u}^2_{2,\td} &= 0, \\
  \de{S} {u^{2}_{2,s}} &= 0, \\
  \de{S}q^{2}_2 &= 0.
\end{align*}
Its unique solution together with the coupling condition is given by
\begin{subequations}
  \label{eq:nearField:k2:2}
\begin{align}
  u^2_{2,\td}(\td,S) &= -{v^2_{2,\td}}(\td)\,
    \mathrm{e}^{-\sqrt{2}\lambda_0 S},  \\
  u^{2}_{2,s}(\td,S) &= 0, \\
  q^{2}_2(\td,S) &= 0.
\end{align}
\end{subequations}

For frequency $0\cdot\omega$ the unique solution is the trivial solution
at least up to order $j=2$, \ie\ the boundary layer disappears.

\subsection{Far field velocity terms}
\label{sec:far-field-velocity}


In the following section we will derive the terms of asymptotic expansion 
for the far field velocity up to order {2}. The resulting 
expressions in frequency $1\cdot\omega$ are exactly the expressions for the 
linear case which {are} derived and analysed in~\cite{Schmidt.Thoens.Joly:2014}. 
The expressions for frequencies $0\cdot\omega$ and $2\cdot\omega$ are only due to the nonlinear
advection term and do not appear for the linear case.
{The general form of the far field equations of any order and in any frequency can be found in the 
Appendix~\ref{sec:appendix:farfield}. }

{
\paragraph{\bf\em Approximation of order $O(\eps^2)$.}
The limit model for the far field velocity in frequency $1\cdot\omega$ is given 
by
\begin{subequations}
  \label{eq:v:0:w1}
    \begin{align}
    \nabla \Div {\vv}^0_1  +\frac{\omega^2}{c^2} {\vv}^0_1  &=  \frac{\imag\omega}{c^2} \vf_0, %
    && \text{ in }\Omega,
    \label{eq:v01:wave} \\
    {\vv}^0_1 \cdot \vn &= 0
    , && \text{ on }\partial\Omega,
    \label{eq:v01:bc}
   \end{align}
\end{subequations}
The far field approximation for frequency $0\cdot\omega$ is given by the stationary incompressible 
Navier-Stokes equations
\begin{subequations}
  \label{eq:v:0:w0}
  \begin{align}
  \label{eq:pv:00}
    \nabla {p}^2_0 +({\vv}^0_0\cdot\nabla){\vv}^0_0
    - \nu_0 \Delta {\vv}^0_0 &= {-\frac{1}{4} \nabla {|\vv^0_1|}^2}
    && \text{in }\Omega \\
    \Div {\vv}^0_0 &= 0 && \text{in }\Omega \\
    \vv^0_0 &= \zerobf, && \text{on } \partial\Omega, 
  \end{align}
  which exhibit a non-linear convection term.  {Here, we have used
    that $\vv_0^0$ is real valued and
    $({\vv}^0_1\cdot\nabla)\overline{\vv^0_1} +
    (\overline{\vv^0_1}\cdot\nabla){\vv}^0_1 = \nabla {|\vv^0_1|}^2$
  since $\scurl\vv^0_1 = 0$}. We see that the unique solution
of~\eqref{eq:v:0:w0} is given by
\begin{align}
   \vv^0_0 &= 0\ , && p^2_0 = - \frac{1}{4}|\vv^0_1|^2\ .
\end{align}
{The stationary limit velocity $\vv^0_0$ vanishes. Note, 
} %
that the stationary pressure terms of order~0 and 1 vanish {as well}, \ie\ $p^0_0 =  p^1_0 = 0$. 
\end{subequations}


\paragraph{\bf\em Approximation of order $O(\eps^3)$.}
The first correcting terms, \ie\ for $j=1$, for frequency $1\cdot\omega$ are 
given by
\begin{subequations}
\label{eq:v:1:w0}
\begin{align}
  \nabla \Div {\vv}^1_1  +\frac{\omega^2}{c^2} {\vv}^1_1  &=  \frac{\imag\omega}{c^2} \vf_1, %
    && \text{ in }\Omega,
    \label{eq:v11:wave} \\    %
    \vv^1_1\cdot\vn &= (1+\imag)\sqrt{\frac{\nu_0}{2\omega}}\frac{c^2}{\omega^2} 
    \partial^2_\td \Div \vv^0_1, 
    && \text{on } \partial\Omega
\end{align}
\end{subequations}
and for frequency $0\cdot\omega$ {the} far field approximation {solves the Stokes system}
\begin{subequations}
\label{eq:v:1:w1}
\begin{align}
  \label{eq:pv:01}
    {\nabla {p}^3_0 - \nu_0 \Delta {{\vv}^1_0}} %
    &= {-\frac{1}{4} \left(
      ({\vv}^0_1\cdot\nabla)\overline{{\vv}^1_1} +
      (\overline{\vv^0_1}\cdot\nabla){{\vv}^1_1} +
      ({\vv}^1_1\cdot\nabla)\overline{{\vv}^0_1} +
      (\overline{\vv^1_1}\cdot\nabla){{\vv}^0_1} \right)}, %
    && \text{in }\Omega\\
    \Div {\vv}^1_0 &= 0,  && \text{in }\Omega \\
    \vv^1_0 &= \zerobf, && \text{on } \partial\Omega\ ,
\end{align}
\end{subequations}
where we have used $\vv^0_0 = 0$. 
The stationary velocity term $\vv^1_0$ is coupled with the stationary pressure 
term $p^3_0$, however, the system is linear.
%

\smallskip
\paragraph{\bf\em Approximation of order $O(\eps^4)$.}
 The next correcting terms, \ie\ for $j=2$, for frequency $1\cdot\omega$ are given by
  \begin{subequations}
  \label{eq:v:2}
  \begin{align}
  \nabla \Div {\vv}^2_1 +\frac{\omega^2}{c^2} {\vv}^2_1 
  &= -\frac{\imag\nu_0\omega^3}{c^4}\vv^0_1 - \frac{\nu_0\omega^2}{c^4}\vf_0,
  && \text{in }\Omega 
  \label{eq:v21:wave} \\
  \vv^2_1\cdot\vn &= %
    \frac{c^2}{\omega^2} \left( (1+\imag)\sqrt{\frac{\nu_0}{2\omega}} %
    \partial_\td^2\Div\vv^1_1%
    + {\frac{\imag\nu_0}{2\omega}}\partial_\td(\kappa\partial_\td\Div \vv^0_1)
    \right), && \text{on } \partial\Omega.
\end{align}
\end{subequations}
where we used~\eqref{eq:v:0:w1} and the fact that $\scurl{\vv}^0_1=0$.
By the assumption on the source function $\scurl \vf_0 = 0$ the term
$\tfrac{\nu_0}{\omega^2}\plcurl\scurl \vv^0_1$ in \eqref{eq:v21:wave} disappears.
{For the system for the frequency $0\cdot\omega$ the far field pressure term 
$p^0_1$ is needed, which is obtained
{\em a-posteriori} from the far field velocity $\vv^0_1$ as}
\begin{align}
  \label{eq:vj1:pj1}
    p^0_1 &= -\frac{\imag c^2}{\omega}\,\Div \vv^0_1\ .
\end{align}
For frequency {$0\cdot\omega$} {far field} second correcting terms are 
\begin{subequations}
\label{eq:pv:02}
  \begin{align}
    \nabla {p}^4_0 
    - \nu_0 \Delta {{\vv}^2_0} &= {
      - (\vv_0^1 \cdot \nabla ) \vv_0^1 -\frac{1}{4} \left(
      ({\vv^0_1}\cdot\nabla)\overline{{\vv}^2_1} +
      ({\vv^1_1}\cdot\nabla)\overline{{\vv}^1_1} +
      ({\vv^2_1}\cdot\nabla)\overline{{\vv}^0_1} \right)} \nonumber \\
                               &\hspace{0.8em} \ 
      -\frac{1}{4} \left(
      (\overline{\vv^0_1}\cdot\nabla){{\vv}^2_1} +
      (\overline{\vv^1_1}\cdot\nabla){{\vv}^1_1} +
      (\overline{\vv^2_1}\cdot\nabla){{\vv}^0_1} \right)
    && \text{in }\Omega \\
  \label{eq:pv:02:mass}
    \Div {\vv}^2_0 &= { -\frac{1}{4 c^2} \big( \vv_1^0 \cdot
                     \overline{\vf_0} + \overline{\vv_1^0} \cdot
                     {\vf_0} \big)} 
                               && \text{in }\Omega \\
    %
    \vv^2_0 &= \zerobf, && \text{on } \partial\Omega
   \end{align}
\end{subequations}
where we used $\scurl \vv^2_1 =0$ which is due to the 
fact that $\scurl \vv^0_1 =0$.

Moreover, the far field velocity at frequency $2\cdot\omega$ is obtained from
\begin{subequations}
\label{eq:v:22}
  \begin{align}
   \nabla \Div {\vv}^2_2 + \frac{4\omega^2}{c^2} \vv^2_2 &= 
    {- \frac{\imag\omega}{c^2} \nabla \big(\vv^0_1\big)^2}
    - \frac{1}{2c^2} \nabla (\vv^0_1 \cdot \vf_0)
    + \frac{\imag}{2\omega} \nabla 
    \big(\Div {{\vv}^{0}_1} \big)^2
    && \text{in }\Omega 
  \label{eq:v22:wave} \\
    %
    {\vv}^2_2\cdot\vn &= 0,    && \text{on }\partial\Omega,
\end{align}
and it follows by \eqref{eq:vj1:pj1} that
\begin{align}
  \label{eq:v:22:p}  
  p^2_2 & = {-\frac{\imag c^2}{2\omega}\Div \vv^2_2 -
          \frac{\imag}{2\omega} \vv_1^0 \cdot \vf_0 -
          \frac{c^2}{2\omega^2} \Big( (\Div \vv_1^0)^2 -
          \frac{\omega^2}{c^2} (\vv_1^0)^2 \Big) }
\end{align}
\end{subequations}

\subsection{Far field pressure up to $2^{\mathrm{nd}}$ order}
\label{sec:far-field-pressure}

As it was mentioned before, the far field approximations in frequency $\omega$ for $j=0,1,2$ are 
exactly the results for the linear problem in~\cite{Schmidt.Thoens.Joly:2014}. Accordingly, we can 
rewrite the equations in terms of the far field pressure with the suitable boundary conditions.
\paragraph{\em Approximation of order $O(\eps^2)$.}
The limit model is given by
  \begin{subequations}
  \label{eq:p:0}
  \begin{align}
    \Delta {p}^0_1 + \frac{\omega^2}{c^2} {p}^0_1 &= \Div {\vf}_0
    && \text{in }\Omega 
   \label{eq:p01:wave} \\
    \nabla p^0_1\cdot\vn &= 0, && \text{on } \partial\Omega. 
    %
  \end{align}
  \end{subequations}
\paragraph{\em Approximation of order $O(\eps^3)$.}
The first correcting terms are given by
  \begin{subequations}
  \label{eq:p:1}
  \begin{align}
    \Delta {p}^1_1 + \frac{\omega^2}{c^2} {p}^1_1 &= \Div\vf_1, 
    && \text{in }\Omega \\
    \nabla p^1_1\cdot\vn &= -(1+\imag)\sqrt{\frac{\nu_0}{2\omega}}
    \partial^2_\td p^0_1, && \text{on } \partial\Omega. 
    %
  \end{align}
  \end{subequations}
\paragraph{\em Approximation of order $O(\eps^4)$.}
 The next correcting terms, \ie\ for $j=2$, for frequency {$\omega$} are given by
  \begin{subequations}
  \label{eq:p:21}
  \begin{align}
  \Delta {p}^2_1 + \frac{\omega^2}{c^2} {p}^2_1
  &= {\Div\vf_2 + }
  \frac{\imag\omega\nu_0}{c^2} \Delta p^0_1,  && \text{in }\Omega \\
  \nabla p^2_1\cdot\vn &= 
    -(1+\imag)\sqrt{\frac{\nu_0}{2\omega}} \partial_\td^2 p^1_1 
    -\frac{\imag\nu_0}{2\omega} \partial_\td(\kappa\partial_\td p^0_1),
    && \text{on } \partial\Omega. 
  \end{align}
  \end{subequations}
  When the far field pressure terms for the frequency $\omega$ are computed we may obtain 
  {\em a posteriori} the far field  velocity terms by
  \begin{align}
  \label{eq:pj1:vj1}
  {\vv}^j_1 &= + \frac{\imag}{\omega} (\vf_j- \nabla {p}^j_1)
    -\frac{\nu_0}{c^2}\nabla p^{j-2}_1, 
  && \text{for } \quad j=0,1,2\ .
  \end{align}
  The far field approximation for frequency $0\cdot\omega$ is given by
  \begin{align}
    \label{eq:p:02} p^2_0 &= -\frac14 \big|\vv^0_1\big|^2, 
  \end{align}
  and in frequency $2\cdot\omega$  by
  \begin{subequations}
  \begin{align}	
  \label{eq:p:22}
    \Delta {p}^2_2 + \frac{4\omega^2}{c^2} {p}^2_2  &= 
  {-\frac{1}{4}\Delta \big({\vv}^0_1\big)^2}
  -\frac{\imag\omega}{c^2}{\Div(p^0_1 \vv^0_1)}
    && \text{in }\Omega\ , \\
    \nabla{p}^2_2\cdot\vn &= 0,    && \text{on }\partial\Omega\ , \\
    {\vv}^2_2 &= 
    {-\frac{\imag}{2\omega}\left(\frac{1}{4}\nabla\big({\vv}^0_1\big)^2 
    + \nabla {p}^2_2 \right)}, 
  && \text{in }\Omega\ .
  \label{eq:v:22:apos}
  \end{align}
  \end{subequations}

\subsection{Deriving the effective systems with impedance boundary conditions}
\label{sec:derivation:effectiveSystems}

In the previous sections we have derived the terms of the asymptotic expansions~\eqref{eq:asympExpan} up to order~2, which we can assemble to obtain pressure and velocity approximations of these orders.
To obtain pressure approximations of order~1 two Helmholtz systems have to be solved, for order~2 these are four Helmholtz systems. %
To obtain velocity approximations of order~1 we need to solve three PDEs, and for order~2 these are already six.  %
In general, the number of terms in the asymptotic expansion increase like $\frac14N^2$ with the order $N$, and, hence, the number of systems to solve.
In this section, we derive the effective systems given in Sec.~\ref{sec:ModelDef} that are written directly for approximative solutions of order 0, 1 and 2.
{The approximative solutions show} the same accuracy as the asymptotic expansions~\eqref{eq:asympExpan} but for a less computational effort as all terms for $j \leqslant N$ would have been computed at once.
Here, the number of systems to solve increases only linearly with $N$ and to obtain pressure and velocity approximations of order~2
only two or three systems, respectively, have to be solved.
The main idea is to combine the equations satisfied by each
{far field} term of~\eqref{eq:asympExpan} and to neglect the next order terms.
{In this way we} obtain equations satisfied by the pressure coefficients
$p_k^{\eps,N}$, where associated velocity coefficients
$\vw_k^{\eps,N}$ are defined afterwords as a function of the pressure
(see Sec.~\ref{sec:deriv-1textrmst-orde} for the first order and in
Sec.~\ref{sec:deriv-2textrmnd-orde} for the second order model),
or equations satisfied by the velocity coefficients
$\vv_k^{\eps,N}$, where associated pressure coefficients
$q_k^{\eps,N}$ are defined afterwords as a function of the velocity
(see Sec.~\ref{sec:deriv-1textrmst-orde-1} and Sec.~\ref{sec:deriv-2textrmst-orde-2} for the first and second order model, respectively).

\subsubsection{Derivation of $1^{\textrm{st}}$ order effective system
  for the pressure}
\label{sec:deriv-1textrmst-orde}

The derivation for the $1^{\textrm{st}}$ order effective system~\eqref{eq:pappr:1} for the pressure is exactly as for the linear case~\cite{Schmidt.Thoens:2014}.
Adding the system~\eqref{eq:p:0} for $p^1_0$ and $\eps$ times the system~\eqref{eq:p:1} for $p^1_1$
we obtain a system for the first order asymptotic expansion $\wt{p}_1^{\eps,1} := \eps^2( p_1^0 + \eps p_1^1 )$
\begin{subequations}
  \label{eq:p:0+eps1}
  \begin{align}
    \Delta \wt{p}_1^{\eps,1} + \frac{\omega^2}{c^2} \wt{p}_1^{\eps,1} &= \Div\vf, 
    && \text{in }\Omega, \\
    \label{eq:p:0+eps1:bc}
    \nabla \wt{p}_1^{\eps,1} \cdot\vn &= - \eps\ (1+\imag)\sqrt{\frac{\nu_0}{2\omega}}
    \partial^2_\td \big( \wt{p}_1^{\eps,1} - \eps^{3} p_1^1), && \text{on } \partial\Omega. 
  \end{align}
\end{subequations}
Neglecting in~\eqref{eq:p:0+eps1:bc} the $O(\eps^{4})$ term and
replacing $\eps \sqrt{\nu_0}$ by $\sqrt{\nu}$ gives~{\eqref{eq:pappr:1} with} 
Wentzel boundary conditions on the domain boundaries.
Now, adding~\eqref{eq:pj1:vj1} for $j=0$ and
$\eqref{eq:pj1:vj1}$ for $j=1$ multiplied by $\eps$ we obtain for the first order asymptotic expansions $\wt{\vv}^{\eps,1}_1 := \vv^0_1 + \eps \vv^1_1$
\begin{align*}
   \wt{\vv}^{\eps,1}_1 &= \frac{\imag}{\omega}(\vf^{\eps,1} - \nabla \wt{p}^{\eps,1} ),
\end{align*}
where $\vf^{\eps,1} = \eps^2(\vf_0 + \eps \vf_1)$. Noting that $\vf = \vf^{\eps,1} + O(\eps^4)$ and if $\wt{p}^{\eps,1}_1 = p^{\eps,1}_1 + O(\eps^4)$ holds
(this is the case if $\frac{\omega^2}{c^2}$ is not a Neumann eigenvalue of $-\Delta$, see ~\cite{Schmidt.Thoens:2014})
and neglecting the $O(\eps^4)$ terms, we find
that~\eqref{eq:wappr:k=1}  defines a first order velocity
approximation~$\vw_1^{\eps,1}$.
Adding the system~\eqref{eq:v:0:w0} for $\vv_0^0$ and $\eps$ times
  the system~\eqref{eq:v:1:w1} for $\vv_0^1$ and neglecting the $O(\eps^4)$ terms 
  we find that~\eqref{eq:vappr:1:k=0} defines a first order approximation stationary velocity
  $\vw_0^{\eps,1}$ that depends on $\vw^{\eps,1}_1$ and incorporates
  with a Lagrange multiplier $q_0^{\eps,3}$ that is $O(\eps^4)$.
%
Finally, we can reconstruct the
pressure and the velocity in time by
\begin{equation}
  \label{eq:reconstruction:1eps:time}
  p^{\eps,1}(t,\vx) = \Re\, p_1^{\eps,1}(\vx) \exp(-\imag \omega t),\qquad
  \vw^{\eps,1}(t,\vx) = \Re\, \sum_{k=0}^1 \vw_k^{\eps,1}(\vx) \exp(-\imag \omega t),
\end{equation}
with the neglected term in the reconstructions~\eqref{eq:reconstruction:1eps:time} being in $O(\eps^4)$.


\subsubsection{Derivation of $2^{\textrm{nd}}$ order effective system
  for the pressure}
\label{sec:deriv-2textrmnd-orde}
Similarly, taking $\eqref{eq:p:0}+\eps
\eqref{eq:p:1}+\eps^2 \eqref{eq:p:21}$, neglecting the $O(\eps^3)$
term and using that $\nu = \eps^2\nu_0$ leads
to the $2^{\textrm{nd}}$ order effective system~\eqref{eq:pappr:2} for the pressure at frequency $\omega$.
As well, we obtain \emph{a posteriori} the far field velocity {approximation} $\vw_1^{\eps,2}$ {defined by~\eqref{eq:wappr:k=1}},
combining~$\eps^j \eqref{eq:pj1:vj1}$ for $j=0$, $j=1$ and $j=2$ and
neglecting the {$O(\eps^5)$} term. 
Using that in the expansion $p_0^0 = p_0^1 = 0$,
taking~\eqref{eq:p:02} and neglecting the $O(\eps^3)$ term leads~to
\begin{equation}
  \label{eq:p:2eps:0}
  p_0^{\eps,2} = -\frac{1}{4} {\left|\vw_1^{\eps,2}\right|^2}\ , 
\end{equation}
{and so to~\eqref{eq:pappr:02}}.
Similarly, {using that $\vw^{\eps,2}_1 = \eps^2\vv^0_1 + O(\eps^3)$, $p^{\eps,2}_1 = \eps^2 p^0_1 + O(\eps^3)$
and $\frac{\omega^2}{c^2}p^{\eps,2} = \eps^2 \Div \vv^0_1 + O(\eps^3)$, the latter being a consequence of~\eqref{eq:pj1:vj1} for $j = 0$ and~\eqref{eq:p:0}, %
we find the $2^{\textrm{nd}}$ order effective system~\eqref{eq:pappr:22} for the pressure contribution at frequency $2\cdot\omega$.}
Then, using the equality~\eqref{eq:v:22:apos} and that $p^{\eps,2}_2 = \eps^{{4}} p^2_2 + O(\eps^{{5}})$ (assuming that $\frac{2\omega}{c}$ is not a Neumann eigenvalue of $-\Delta$) 
and $\frac{\imag}{\omega}(\vf - \nabla p^{\eps,1}) = \eps^2 \vv^0_1 + O(\eps^{5})$ we obtain the equation~\eqref{eq:wappr:k=2} for the velocity approximation $\vw^{\eps,2}_2$ in terms of $p^{\eps,2}_1$
and $p^{\eps,2}_2$.
Adding the system~\eqref{eq:v:0:w0} for $\vv_0^0$, $\eps$ times
  the system~\eqref{eq:v:1:w1} for $\vv_0^1$ and $\eps^2$ times the
  system~\eqref{eq:pv:02} we find that
  \eqref{eq:vappr:2:k=0} defines a first order approximation stationary velocity
  $\vw_0^{\eps,2}$ with a Lagrange multiplier $q_0^{\eps,4} = 
  p_0^{\eps,2} + O(\eps^5)$.
Finally, we can reconstruct the pressure and the velocity in time by
\begin{equation}
  \label{eq:reconstruction:2eps:time}
  p^{\eps,2}(t,\vx) = \Re\, \sum_{k=0}^2 p_k^{\eps,2}(\vx) \exp(-\imag
  k \omega t),\qquad
  \vw^{\eps,2}(t,\vx) = \Re\, \sum_{k=0}^2 \vw_k^{\eps,2}(\vx)
  \exp(-\imag k \omega t)\ ,
\end{equation}
with the neglected term in the reconstructions~\eqref{eq:reconstruction:2eps:time} being in $O(\eps^5)$.

\subsubsection{Derivation of $1^{\textrm{st}}$ order effective system for
  the velocity}
\label{sec:deriv-1textrmst-orde-1}

Taking $\eqref{eq:v:0:w1}+\eps \eqref{eq:v:1:w0}$ and neglecting the
$O(\eps^4)$ term leads to the $1^{\textrm{st}}$ order effective system~\eqref{eq:vappr:1} for the velocity component to the frequency~$\omega$.

Now, adding~$\eps^4\eqref{eq:v:0:w0}$ and $\eps^5\eqref{eq:v:1:w1}$ and using that $\vv^0_0 = 0$ we find %
that $\wt{\vv}^{\eps,1}_0 := \eps^3 \vv^1_0$, $\wt{p}^{\eps,1}_0 := \eps^4 (p^2_0 + \eps p^3_0)$ solve %
\begin{align*}
   \nabla \wt{p}^{\eps,1}_0 - \nu \Delta\wt{\vv}^{\eps,1}_0 &= -\frac14 \left( (\wt{\vv}^{\eps,1}_1\cdot \nabla) \overline{\wt{\vv}^{\eps,1}_1} + \overline{\wt{\vv}^{\eps,1}_1}\cdot \nabla) \wt{\vv}^{\eps,1}_1 \right)
   + \frac{\eps^6}{4}\left((\vv^1_1\cdot\nabla)\overline{\vv^1_1} + (\overline{\vv^1_1}\cdot\nabla)\vv^1_1\right), &&\text{ in } \Omega\ ,\\
   \Div \wt{\vv}^{\eps,1}_0 &= 0, &&\text{ in } \Omega\ ,\\
   \wt{\vv}^{\eps,1}_0 &= \zerobf, &&\text{ on } \partial\Omega\ .
\end{align*}
Neglecting the $O(\eps^6)$ term on the right hand side leads to~\eqref{eq:vappr:1:k=0}. %
Finally, we get similarly the equality~\eqref{eq:pappr:N1} for the {\em a-posteriori} computed pressure approximation $q^{\eps,1}_1$, where~\eqref{eq:vj1:pj1} and 
\begin{align*}
   p^1_1 &= -\frac{\imag c^2}{\omega} \Div \vv^1_1
\end{align*}
are used.

\subsubsection{Derivation of $2^{\textrm{nd}}$ order effective system for
  the velocity}
\label{sec:deriv-2textrmst-orde-2}

The derivation of the systems~\eqref{eq:vappr:2} for $\vv^{\eps,2}_1$ and~\eqref{eq:vappr:2:k=0} for $\vv^{\eps,2}_0$ is similar to the respective first order systems as well 
as equality~\eqref{eq:pappr:N1} for $q^{\eps,2}_1$. %
The system~\eqref{eq:vappr:22} for the velocity contribution $\vv^{\eps,2}_2$ at frequency $2\cdot\omega$ is a direct consequence of~\eqref{eq:v:22}
and the equation\eqref{eq:pappr:22} for the pressure $q^{\eps,2}_2$ is a direct consequence of~\eqref{eq:v:22:p}.
Then, the system~\eqref{eq:vappr:2:k=0} for the velocity contribution $\vv^{\eps,2}_0$ at frequency $0\cdot\omega$ 
is derived similarly to~\eqref{eq:vappr:1:k=0} of order~1 using the systems~\eqref{eq:v:1:w1} and~\eqref{eq:pv:02}. %



\begin{figure}[tb]
\centering
\subfigure[For the viscosity $\nu=3.6\cdot10^{-3}$, $\|\vf\|_\infty=0.163$, 40 periods in the 
exact model.]
{  \parbox[t]{3cm}{\includegraphics[width=3cm]{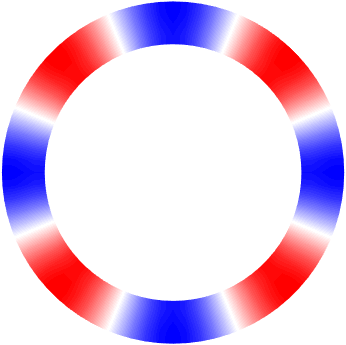} 
\\\centering{Order $N=0$}}
     \hfill
  \parbox[t]{3cm}{\includegraphics[width=3cm]{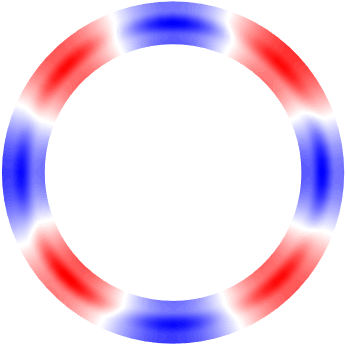} 
\\\centering{Order $N=1$}}
     \hfill
  \parbox[t]{3cm}{\includegraphics[width=3cm]{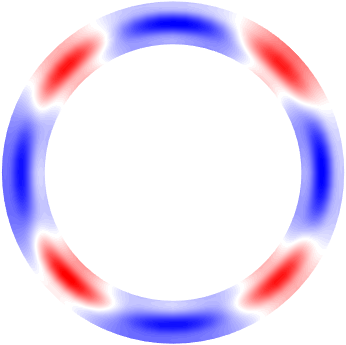} 
\\\centering{Order $N=2$}} 
     \hfill
  \parbox[t]{3cm}{\includegraphics[width=3cm]{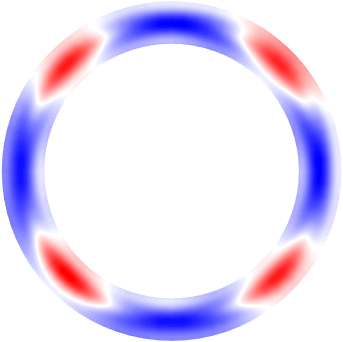} 
\\\centering{Exact model}}
  \hfill
  \parbox[t]{3cm}{\includegraphics[width=3cm]{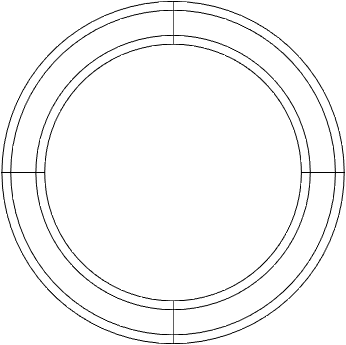} 
  \\\centering{Mesh}   } 
}
\subfigure[For the viscosity $\nu=4\cdot10^{-4}$, $\|\vf\|_\infty=0.018$, 25 periods in the exact 
model.]
{  \parbox[t]{3cm}{\includegraphics[width=3cm]{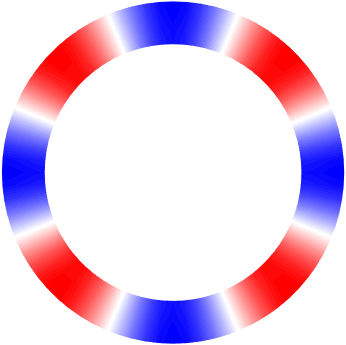} 
\\\centering{Order $N=0$}}
     \hfill
  \parbox[t]{3cm}{\includegraphics[width=3cm]{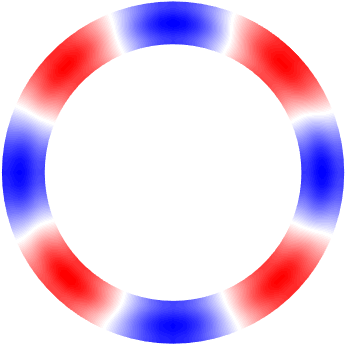} 
\\\centering{Order $N=1$}}
     \hfill
  \parbox[t]{3cm}{\includegraphics[width=3cm]{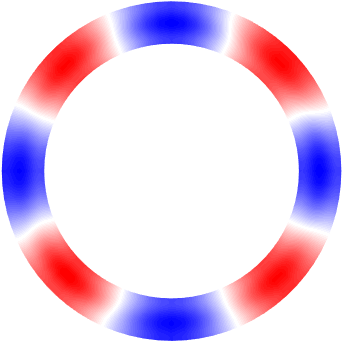} 
\\\centering{Order $N=2$}} 
     \hfill
  \parbox[t]{3cm}{\includegraphics[width=3cm]{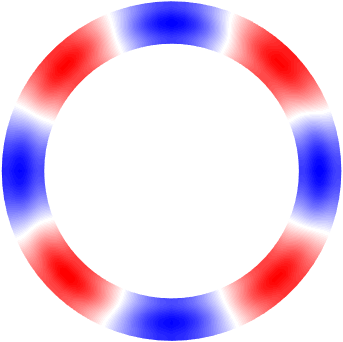} 
\\\centering{Exact model}}
  \hfill
  \parbox[t]{3cm}{\includegraphics[width=3cm]{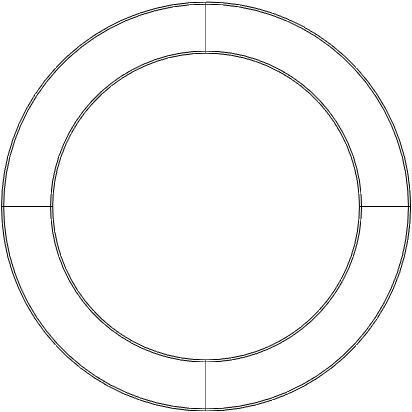} 
  \\\centering{Mesh}   }
}
\caption[Comparison]{Comparison the quasi-stationary state solution for the pressure at the end 
of the period $T=2\pi/\omega$ of
  the approximate models of order $N=0, 1, 2$ to the exact pressure
  (frequency of the excitation $\omega=15$).  The mesh
  resolving the boundary layers used in the FEM of higher order is
  shown in the right subfigure.}
    \label{fig:pressure}
\end{figure}

\section{Numerical results}
\label{sec:numerics}

We verify the derived approximative models with impedance boundary conditions on a ring domain $\Omega$ centered at $(0,0)$
whose inner radius is $R_1 = 1.5$ and outer radius is $R_2 = 2.0$. %
For this we choose several values for $\eps$, where the viscosity $\nu = \eps^2$ (\ie, $\nu_0 = 1$)
and the source $\vf$ takes the decomposition $\vf(t,\vx) = \left(\eps^2 \vf_0(\vx) + \eps^3 \vf_1(\vx)\right)\cos(\omega t)$ 
with $\omega = 15$, where the dominating part $\eps^2\vf_0$ is $\scurl$-free. More precisely, we take 
$\vf_0 = \Re \nabla p_0$ with 
\begin{align*}
  p_0(\vx) = \left( \big( Y_{\lambda-1}(k R_1)- Y_{\lambda+1}(k R_1) \big) J_\lambda(kr(\vx)) 
  + \big( J_{\lambda+1}(k R_1)- J_{\lambda-1}(k R_1) \big) Y_\lambda(kr(\vx)) \right)
  \mathrm{e}^{\imag \lambda \phi(\vx)}\ ,
\end{align*}
with the polar coordinates $(r,\phi)$ in the ring, $\lambda = 4$ and $k = 2.28945$ computed numerically such that 
the Neumann trace $\nabla p_0\cdot\vn = 0$ on the boundary $\partial\Omega$ of the ring.
In this way, $p_0 \in H^1(\Omega)$ is solution of the Helmholtz equation
\begin{align*}
  \Delta p + k^2 p &= 0, && \text{in }\Omega\ , \\
  \nabla p \cdot \vn &=0,   && \text{on } \partial\Omega\ .
\end{align*}
Hence, the normal component of the source $\vf_0$ vanishes on $\partial\Omega$. %
As the tangential component of $\vf_0$ does not vanish we use the formulations with additional terms $\vf\cdot\vn^\bot$
that are given in Appendix~\ref{sec:appendix:IBC:source}. %
Moreover, the second term of the source is a bubble function {with} 
$\vf_1(\vx) = (R_1^2 - r(\vx)^2)(R_2^2 - r(\vx)^2)\binom{1}{1}$
that vanishes in both components on $\partial\Omega$.

We have computed numerically approximative solutions of different order by high order finite elements with curved cells
using the numerical C++ library Concepts~\cite{Frauenfelder.Lage:2002,Schmidt.Kauf:2009,conceptsweb},
where we use the formulation for the pressure. To estimate the modelling error of these approximative solution we compute numerically 
a reference solution in time-domain using a {modified} Crank-Nicolson {scheme} 
{in which} the nonlinear advection terms are discretized explicitly (see~\cite{Tone:2004,Yang:2009} for similar schemes for incompressible fluids).
The time-domain formulation with time step $\Delta t > 0$ is in both variables, 
the pressure and the velocity, and given by
\begin{align*}
  \frac{\vv^{\ell+1}-\vv^\ell}{\Delta t} + (\vv^\ell \cdot \nabla)\vv^\ell
  -\nu \Delta \vv^{\ell+\frac12} + \nabla p^{\ell+\frac12} &= \vf^{\ell+\frac12}\ ,&& \text{ in } \Omega\ , \\
  \frac{p^{\ell+1}-p^\ell}{\Delta t} + c^2 \Div \vv^{\ell+\frac12} 
  + \Div(p^{\ell} \vv^{\ell}) &=0\ ,&& \text{ in } \Omega\ ,\\[0.5em]
  \vv^{\ell+1} &= \zerobf\ ,&& \text{ on } \partial\Omega\ ,
\end{align*}
where $(\vv^\ell, p^\ell)$ is a numerical approximation to $(\vv(\ell\Delta t,\cdot), p(\ell\Delta t,\cdot))$,
$\vf^\ell := \vf(\ell\Delta t,\cdot)$ and $\vv^{\ell+\frac12}$, $p^{\ell+\frac12}$, $\vf^{\ell+\frac12}$ denote the averages
\begin{align*}
  \vv^{\ell+\frac12} &:= \tfrac{1}{2}(\vv^{\ell+1}+\vv^\ell)\ , &
  p^{\ell+\frac12} &:= \tfrac{1}{2}(p^{\ell+1}+p^\ell)\ &
  \vf^{\ell+\frac12} &:= \tfrac{1}{2}(\vf^{\ell+1}+\vf^\ell)\ .
\end{align*}
As initial velocity and pressure we use the solution of the linear system (without the nonlinear advection terms) and simulate for $15$ periods 
to $50$ depending on $\eps$ to obtain an accurate approximation to the quasi-stationary solution that is periodic in $t$. %
To resolve the boundary layers of order $\eps$ numerically we use the {\em hp}-adaptive strategy of Schwab and Suri~\cite{Schwab.Suri:1996},
where we use a mesh with curved cells with a high aspect ratio (see the meshes in Fig.~\ref{fig:pressure}) where the size normal to the boundary 
behaves linear in $\eps$ (or $\sqrt{\nu}$), $\sqrt{\omega}$ and the polynomial order $p$.
In the experiments we have used a uniform polynomial order $p = 8$ and a time steps $\Delta t$ between $2\cdot 10^{-4}$ for $\eps = 10^{-2}$ and
$2\cdot 10^{-3}$ for $\eps = 10^{-1}$. %
Even so not necessary, we use the same mesh and polynomial order for the approximative models of order $N = 0, 1, 2$. Note, that the computation of
the reference solution is by far more expensive than the computation of the approximative models.

\begin{figure}[tb]
  \centering
\begin{tikzpicture}
\begin{axis}[
    width=10cm, height=6cm,
    ybar,
    enlargelimits=0.15,
    ymode=log,
    log origin=infty,
    ymax = 1e-1,
    legend style={at={(0.5,-0.25)},
    anchor=north,legend columns=-1},
    ylabel={pressure level},
    xtick=data,
    xlabel={frequency},
    xticklabels={0, $\omega = 15$, $2\cdot\omega = 30$, $3\cdot\omega = 45$}
    ] 
%
\addplot coordinates {(0, 4.3632e-003) (15, 33.1884e-003) (30, 605.4373e-006) (45, 76.0945e-006)};
\addplot coordinates {(0, 4.2328e-003) (15, 32.9257e-003) (30, 549.8469e-006) };
\legend{exact model, approximative model of order 2}
\end{axis}
\end{tikzpicture}
\caption{$L_2(\Omega)$-norm bar chart comparing the different frequencies modes of the 
solution $p(\vx)$ of the original nonlinear system and the approximative model in frequency domain of order 2
for viscosity $\nu=3.6\cdot 10^{-2}$ and $\|\vf\|_\infty=1.72$.}
\label{fig:L2norm_of_modes}
\end{figure}
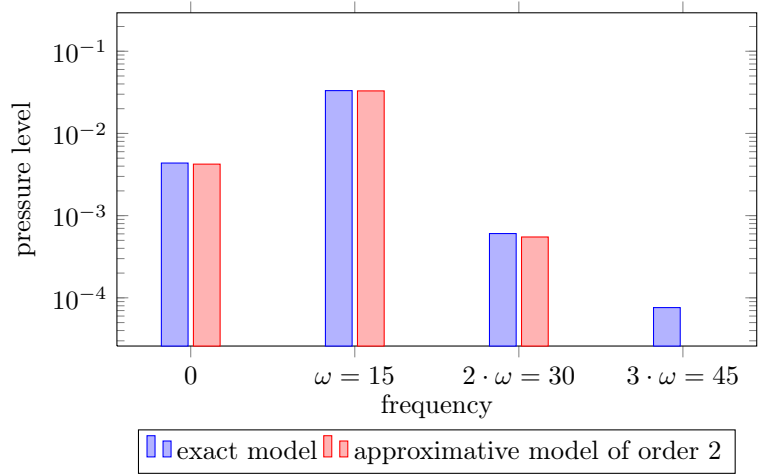

\begin{figure}[tb]
  \centering
  \begin{tikzpicture}[scale=0.96]
    \pgfplotsset{ legend style={
        at={(0.05,0.83)},
        font=\footnotesize, anchor=west} %
    } %
    \begin{loglogaxis}[%
      width=12cm, height=7cm,
      xlabel=$\eps$,
      ylabel=modelling error,
      xmin=1e-2,xmax=1e-1,
      ymin=1e-2,ymax=1e0,
      legend style={draw=none} 
      ] 
      \addplot+[only marks] table [x = eps, y = Order0] {pModError.dat};
      \addlegendentry{Order $N=0$};
      \addplot+[only marks] table [x = eps, y = Order1] {pModError.dat};
      \addlegendentry{Order $N=1$};
      \addplot+[only marks, mark=diamond*, violet, mark options={fill=violet},mark size=3pt] 
        table [x = eps, y = Order2] {pModError.dat};
      \addlegendentry{Order $N=2$};
      \addplot [blue]
      table [
        x=eps, 
        y={create col/linear regression={
        y=Order0, 
        variance list={100,100,1,1,1,1}
        }}
      ] {pModError.dat}
        coordinate [pos=0.2] (A) 
        coordinate [pos=0.5] (B)
      ;
     \draw (B) -| (A)  
        node [pos=0.75,anchor=east] {0.98} 
     ;
      \addplot [red]
      table [
        x=eps, 
          y={create col/linear regression={
          y=Order1,
        variance list={100,100,100,1,1,1}
          }}
        ] {pModError.dat}
        coordinate [pos=0.5] (A) 
        coordinate [pos=0.7] (B)
      ;
     \xdef\slope{\pgfplotstableregressiona} 
     \draw (B) -| (A)  
        node [pos=0.75,anchor=east] {1.98} 
     ;
      \addplot [violet]
      table [
        x=eps, 
          y={create col/linear regression={
          y=Order2,
        variance list={100,100,100,100,100,1,1}
        }}
        ] {pModError.dat}
        coordinate [pos=0.6] (A) 
        coordinate [pos=0.8] (B)
      ;
     \xdef\slope{\pgfplotstableregressiona} 
     \draw (A) -| (B)  
        node [pos=0.75,anchor=west] {\num[round-mode=figures,round-precision=3]{\slope}} 
     ;
    \end{loglogaxis}
  \end{tikzpicture}
\caption{The relative modelling error 
  {$\|p-p^{\eps,N}\|_{L^2\big((0,T),L^2(\Omega)\big)}/\|p\|_{L^2\big((0,T),L^2(\Omega)\big)}$ 
  for \mbox{$N=0,1,2$} as a function of the parameter $\eps$.}}
\label{fig:error}
\end{figure}
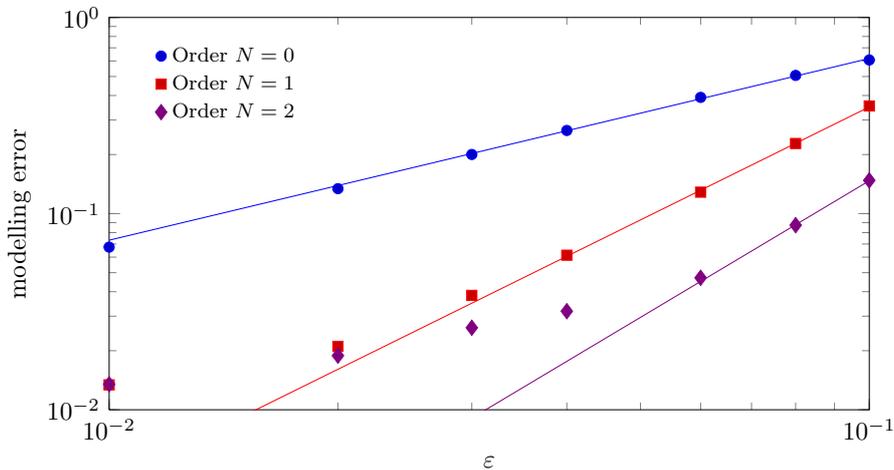

In Fig.~\ref{fig:pressure} the approximative pressures distributions
\begin{align}
   p^{\eps,N}(t,\vx) = \Re \sum_{k=0}^{\lceil\frac{N+1}{2}\rceil} p^{\eps,N}_k(\vx) \mathrm{e}^{-\imag k \omega t}\ ,
\end{align}
that are composed of the modes $p^{\eps,N}_k$ are shown for $N = 0, 1, 2$ in comparison with the reference solution that is obtained in time domain. 
Depending on the magnitude of the viscosity and the source a good agreement is achieved with a high enough order of the approximative solution
that is $N = 0$ or $N = 2$ in the two examples.

Using the inverse Fourier transform of the reference pressure in the respective last period we have computed the $L^2(\Omega)$-norm of the contributions 
to the frequencies $0$, $\omega$, $2\,\omega$ and $3\,\omega$ and obtained a very well agreement with the contributions of the approximative solution of second order 
(see Fig.~\ref{fig:L2norm_of_modes}).

Finally, we have studied the modelling error of the approximative models of order $N = 0,1,2$ in dependence of the parameter $\eps$ in the $L^2(\Omega)$-norm in the last period
of the in time-domain computed reference solution (see Fig.~\ref{fig:error}).
We clearly observe convergence orders of $1$ and $2$ of the relative $L^2$-modelling error for the approximative model of order~$0$ and $1$, which are both linear and 
have only contributions at the frequency $\omega$. For the approximative model of order~$2$ that has non-linear contributions for the frequencies $0$ and $2\cdot\omega$
we find numerically a convergence rate larger than two and much lower error levels for the considered values of $\eps$. %

\section{Conclusion}

For the acoustic wave-propagation in the presence of viscous boundary layers and frequency interaction due to 
nonlinear advection terms approximative models up to order~2 with impedance boundary conditions have been introduced. %
They are based on a multiscale expansion and multiharmonic analysis of the compressible nonlinear Navier-Stokes equations 
for small sound amplitudes of $O(\eps^2)$ and viscosities of $O(\eps^2)$, where $\eps$ is a small parameter. %
In the approximative models the contributions to the excitation frequency $\omega$ and its harmonics can be 
computed sequentially after each other. %
As the approximative models are for macroscopic pressure or velocity fields no adaptive mesh refinement is necessary,
which implies in the original model of the Navier-Stokes equations a reduction of time steps in explicit and semi-implicit schemes. %
In numerical experiments using finite element discretisation of the frequency domain approximative model and the original 
instationary compressible Navier-Stokes equations a good agreement has been shown as well as a convergence of the approximative
solutions to the reference solution. %

The derivation is for small sound amplitudes and it is of interest to extent the results to higher sound amplitudes than $O(\eps^2)$ 
which exhibit higher frequency interaction in the viscous layers. %
Moreover, the nonlinear Navier-Stokes equations and frequency interaction is of high interest for the modelling of liners
where for periodically perforated plates the method of surface homogenization has been developed to obtain approximative impedance
transmission conditions~\cite{Bonnet.Drissi.Gmati:2004,Delourme.Schmidt.Semin:2016}.

\section*{Acklowledgements}

The authors gratefully acknowledge the financial support by the research center {\sc Matheon} through the 
Einstein Center for Mathematics Berlin (project MI--2) and are thankful to the fruitful exchange with the DLR Berlin.

\appendix
\section*{Appendix A}
\setcounter{section}{1}
\renewcommand\thesection{\Alph{section}}
\label{sec:appendix}
\addcontentsline{toc}{section}{\numberline {A} The far and near field equations to any order}%

\subsection{Deriving the far field equations}
\label{sec:appendix:farfield}

The far field terms will be defined in physical coordinates in the
whole domain $\Omega$ where we assume $\partial\Omega$ to be
$C^\infty$. %
Inserting the expansion~\eqref{eq:asympExpan} into the
system~\eqref{eq:navier.stokes:eps} for $\vv^{\eps,M}$, $p^{\eps,M}$
for a particular coordinate $\vx \in \Omega$ and letting $\eps$
tend to zero, the near field terms concentrate closer and
closer to the wall and vanish on $\vx$. %
Collecting terms of the same order in $\eps$ results in the far field
  equations:
\begin{subequations}
\label{eq:farfield}
\begin{multline}
\label{eq:moment:k1}
  -\imag k \omega {\vv}_k^j +  \nabla {p}_k^j  
  = {\vf_j} \cdot \delta_{k=1} 
    + \nu_0\Delta {\vv}_k^{j-2}  \\
  -\frac 12 \sum_{\ell=0}^{j-2} \left(
  \sum_{m=0}^{k}  ({\vv}_m^\ell\cdot\nabla){\vv}_{k-m}^{j-\ell-2} + 
  \sum_{m=k}^{M} ({\vv}_m^\ell\cdot\nabla)\overline{{\vv}_{m-k}^{j-\ell-2}} +
  \sum_{m=0}^{M-k} (\overline{{\vv}_m^\ell}\cdot\nabla){\vv}_{m+k}^{j-\ell-2} 
  \cdot \delta_{k\neq 1} \right), 
\end{multline}
\vspace{-2em}
\begin{multline}
\label{eq:mass:k1}
  -\imag k \omega {p}_k^j + c^2\,\Div {\vv}_k^j =
  -\frac 12 \sum_{\ell=0}^{j-2} \left( 
  \sum_{m=0}^{k}  {\vv}_m^\ell\cdot\nabla {p}_{k-m}^{j-\ell-2} + {p}_m^\ell 
\Div{\vv}_{k-m}^{j-\ell-2} \right. \\
  + \sum_{m=k}^{M} \left. {\vv}_m^\ell\cdot\nabla \overline{{p}_{m-k}^{j-\ell-2}} + {p}_m^\ell 
\Div\overline{{\vv}_{m-k}^{j-\ell-2}} +
  \sum_{m=0}^{M-k} \left(\overline{{\vv}_m^\ell}\cdot\nabla {p}_{m+k}^{j-\ell-2} + 
\overline{{p}_m^\ell} 
\Div{\vv}_{m+k}^{j-\ell-2}\right) \cdot \delta_{k\neq 1} \right)
\end{multline}
\end{subequations}
for $k \in \IN^+$, where $\vv^{-1}_k = \vv^{-2}_k = {\mathbf 0}$,  
and ${\vf_j}=0$ for $j>1$.
The far field equations will be completed by boundary conditions,
which are specified in Sec.~\ref{sec:Main} for $j=0,1,2$.

\subsection{Deriving the near field equations} 
\label{sec:appendix:nearfield}

The following near field equations in local coordinates derived under a condition, 
that the near field expansion~\eqref{eq:BL:expan} inserted 
into~\eqref{eq:navier.stokes:eps} 
leaves a residual as small as possible in powers of~$\eps$, which is at least of order $\eps^{N+1}$ 
\vspace{-2em}
\begin{subequations}
\label{eq:nf}
\begin{multline} 
  \label{eq:nfmomt}
  \imag k\omega{u}^j_{k,\td} +\nu_0\de{S}^2 {u}^j_{k,\td}
   = \sum_{\ell=1}^3 C_{\ell} (\vu^{j-\ell}_k)  
  +  \sum_{\ell=0}^2 %
  \begin{pmatrix}
    2 \\
    \ell \end{pmatrix} 
(-\kappa S)^\ell \de{\td}q^{j-\ell} \\
  + \frac 12 \sum_{\ell=1}^{4} \sum_{i=0}^{j-\ell} \left(
 \sum_{m=0}^{k} E_\ell( {\vu^i_m},\vu^{j-i-\ell}_{k-m}) + 
 \sum_{m=k}^{M} E_\ell( {\vu^i_m},\overline{\vu^{j-i-\ell}_{m-k}} ) +
 \sum_{m=0}^{M-k} E_\ell( \overline{\vu^i_m},\vu^{j-i-\ell}_{k+m} ) 
  \cdot \delta_{k\neq 0} \right)
\end{multline} 
\vspace{-2em}
\begin{multline} 
  \label{eq:nfmoms}
  \de{S}q^{j}_k = 
  \imag k\omega{u}^{j-1}_{k,s} +\nu_0\de{S}^2 {u}^{j-1}_{k,s} 
  ({\vu^{j}_k}^\top)
   - \sum_{\ell=1}^3 \left(C_{\ell} ({\vu^{j-1-\ell}_k}^{\bot}) +
  \begin{pmatrix}
    3 \\
    \ell \end{pmatrix} 
  (-\kappa S)^\ell \de{\td}q^{j-\ell}_k \right) \\
   - \frac 12 \sum_{\ell=1}^{4} \sum_{i=0}^{j-1-\ell} \left(
 \sum_{m=0}^{k} E_\ell( {\vu^i_m},{\vu^{j-1-i-\ell}_{k-m}}^{\bot}) + 
 \sum_{m=k}^{M} E_\ell( {\vu^i_m},{\overline{\vu^{j-1-i-\ell}_{m-k}}}^{\bot} ) +
 \sum_{m=0}^{M-k} E_\ell( {\overline{\vu^i_m},\vu^{j-1-i-\ell}_{k+m}}^{\bot} ) 
  \cdot \delta_{k\neq 0} \right)
\end{multline} 
\vspace{-2em}
\begin{multline} 
\label{eq:nfmass}
  \de{S} {u^{j}_{k,s}} =
   - \de{\td} u^{j-1}_{k,\td}
   +  \kappa (S\de{S} {u^{j-1}_{k,s}} + {u^{j-1}_{k,s}}) 
  + \imag k\omega( {q^{j-1}_{k}} -\kappa S {q^{j-2}_{k}} )\\
  - \frac {1}{2c^2} \sum_{\ell=1}^{2}  \sum_{i=0}^{j-1-\ell} \left(
 \sum_{m=0}^{k} G_\ell( {\vu^i_m},q^{j-1-i-\ell}_{k-m}) + 
 \sum_{m=k}^{M} G_\ell( {\vu^i_m},\overline{q^{j-1-i-\ell}_{m-k}} ) +
 \sum_{m=0}^{M-k} G_\ell( \overline{\vu^i_m},q^{j-1-i-\ell}_{k+m} ) 
  \cdot \delta_{k\neq 0} \right)
\end{multline}
\end{subequations}
where the coefficients are following
\begin{align*}
   C_1(\vu) &= \kappa\big(3\,\imag k \omega S + 3\nu_0 S \de{S}^2 + 
\nu_0\de{S}\big) u_\td, \\
   C_2(\vu) &= -\nu_0\de{\td}^2u_\td -\kappa^2 \big(3\,\imag k\omega S^2 + 
    3\nu_0 S^2\de{S}^2  + \nu_0 (2 S\de{S} -1) \big) u_\td 
    +\nu_0(2\kappa\de{\td}+\kappa')  u_s, \\
   C_3(\vu) &= \kappa^3\big(\imag k \omega S^3 + \nu_0 S (S^2\de{S}^2
    +  S\de{S} - 1 )\big)u_\td + \nu_0 \big(S (\kappa \de{\td}^2  
    - \kappa' \de{\td})u_\td - 2\kappa^2 S\de{\td}u_s\big).
\end{align*}
Coefficients related to the nonlinear terms in the momentum equation 
given by
\begin{align*}
 E_1(\vu,\vv) &= u_s \de{S}v_\td, \\
 E_2(\vu,\vv) &= u_\td \de{\td}v_\td -\kappa ( 3 S u_s 
    \de{S}v_\td + u_\td v_s), \\
 E_3(\vu,\vv) &= \kappa S ( -2u_\td \de{\td}v_\td + 3\kappa S u_s \de{S}v_\td 
  + 2\kappa S u_\td v_s), \\
 E_4(\vu,\vv) &= \kappa^2 S^2 ( u_\td \de{\td}v_\td
  -\kappa S u_s \de{S}v_\td  -\kappa  u_\td v_s ).
\shortintertext{
Those, for the near field continuity equation}
 G_1(\vu,q) &=  u_S \de{S} q + q \de{S} u_s, \\
 G_2(\vu,q) &= u_\td \de{\td} q + q\de{\td}u_\td -\kappa \big( S u_S \de{S} 
q + q ( S\de{S} u_s + u_s) \big).
\end{align*}

\subsection{Impedance boundary conditions for the far field with the source on the 
boundary}
\label{sec:appendix:IBC:source}

In case if the source function $\vf$ does not disappear on the boundary impedance boundary conditions contain additional terms in the frequency of the excitation mode $1\cdot\omega$. For the far field pressure they are
\begin{subequations}
 \begin{align}
  \nabla {p}^{\eps,0}_1 \cdot \vn  &= \vf\cdot\vn, \\
  \nabla {p}^{\eps,1}_1 \cdot \vn + (1+\imag)\sqrt{\frac{\nu}{2\omega}}\partial_\td^2 
  {p}^{\eps,1}_1  &= \vf \cdot \vn-(1+\imag)\sqrt{\frac{\nu}{2\omega}}
  \partial_\td(\vf\cdot\vn^\bot), \\
  \nabla {p}^{\eps,2}_1 \cdot \vn + (1+\imag)\sqrt{\frac{\nu}{2\omega}}
    \partial_\td^2 {p}^{\eps,2}_1 
    +\frac{\imag\nu}{2\omega}\partial_\td(\kappa\partial_\td{p}^{\eps,2}_1) &= \\
    \Big(1+ \frac{\imag\omega\nu}{c^2}\Big)\vf \cdot \vn 
    -(1+\imag)\sqrt{\frac{\nu}{2\omega}} \partial_\td(\vf\cdot\vn^\bot)
    &- \frac{\imag\nu}{2\omega} \partial_\td(\kappa \vf\cdot\vn^\bot)
    - \frac{\imag\nu}{\omega}\plcurl\scurl\vf\cdot\vn. \nonumber
 \end{align}
\end{subequations}
and for the far field velocity
\begin{subequations}
 \begin{align}
    {\vv}^{\eps,0}_1\cdot\vn &= 0, \\
    {\vv}^{\eps,1}_1\cdot\vn - (1+\imag)\frac{c^2}{\omega^2}
    \sqrt{\frac{\nu}{2\omega}} \partial_\td^2\Div{\vv}^{\eps,1}_1  
    &= \frac{(\imag-1)}{\omega} \sqrt{\frac{\nu}{2\omega}}
    \partial_\td(\vf\cdot\vn^\bot) \\
    {\vv}^{\eps,2}_1\cdot\vn - \frac{c^2}{\omega^2} \Big(
    (1+\imag)\sqrt{\frac{\nu}{2\omega}} \partial_\td^2\Div{\vv}^{\eps,2}_1
    &+ \frac{\imag\nu}{2\omega}
    \partial_\td(\kappa\partial_\td\Div {\vv}^{\eps,2}_1)\Big)  \\
    &= \frac{(\imag-1)}{\omega} \sqrt{\frac{\nu}{2\omega}}
    \partial_\td(\vf\cdot\vn^\bot)
    - \frac{\nu}{2\omega^2}\partial_\td(\kappa\,\vf\cdot\vn^\bot) \nonumber
 \end{align}
\end{subequations}

The systems for frequencies $0\cdot\omega$ or $2\cdot\omega$ do not change.

\subsection{Reducing the order of derivation in the system for frequency $2\cdot\omega$}
\label{sec:appendix:p22}

Using vector calculus identities and the relation between far field pressure and velocity at frequency $1\cdot\omega$ we can rewrite equation~\eqref{eq:pappr:22:PDE} as
  \begin{equation}
      \Delta {p}^{\eps,2}_2 + \frac{4\,\omega^2}{c^2} {p}^{\eps,2}_2
      = \frac{1}{2\omega^2} \mathbf{tr} \big(\mathbf{J}(\vf - \nabla
      p^{\eps,2}_1 )^\top \mathbf{J}(\vf - \nabla p^{\eps,2}_1 ) \big)
      + \frac{1}{c^2} \left( \frac{3}{2} \big(\vf - \nabla p^{\eps,2}_1\big) \cdot
        \nabla p^{\eps,2}_1 + \frac{\omega^2}{c^2}
        \big(p^{\eps,2}_1\big)^2 \right) \ ,
  \end{equation}
where $\vJ$ denotes the Jacobi matrix. Here, only second derivatives of $p^{\eps,2}_1$ appear, but no third derivatives.

\ifx\undefined\allcaps\def\allcaps#1{#1}\fi\def\cprime{$'$}

\end{document}